\documentclass[a4paper,leqno]{article}
\usepackage{amssymb}
\usepackage{graphicx}
\usepackage{amsmath}
\usepackage{amsfonts}
\usepackage{amsthm}
\usepackage{mathrsfs}
\usepackage{dsfont}
\usepackage{indentfirst}
\usepackage{esint}
\usepackage{hyperref}
\usepackage{yhmath}

\numberwithin{equation}{section}

\newcommand{\Sec}{\mathrm{Sec}}
\newcommand{\diam}{\mathrm{diam}}
\newcommand{\vol}{\mathrm{vol}}

\title{\Large \bf \boldmath\ \\ Geometric analysis aspects of infinite semiplanar graphs with nonnegative curvature} 

\author{\large  Bobo Hua$^\ast$\ \ \ \  J\"{u}rgen Jost$^{\dag}$ \ \ \ \ Shiping Liu$^{\ddag}$} 

\date{}

\begin{document}
\par
\maketitle

\renewcommand{\thefootnote}{\fnsymbol{footnote}}

\footnotetext{\hspace*{-5mm} \begin{tabular}{@{}r@{}p{13.4cm}@{}}
$^\ast$ & Max Planck Institute for Mathematics in the Sciences, 04103 Leipzig,  Germany.\\
&{Email: bobohua@mis.mpg.de} \\
$^{\dag}$ & supported by ERC Advanced Grant FP7-267087.\\
& Max Planck Institute for Mathematics in the Sciences, 04103 Leipzig,
 Germany, and \\
&Department of Mathematics and Computer Science, University of
Leipzig, 04109 Leipzig, Germany\\

&{Email: jost@mis.mpg.de} \\
$^{\ddag}$ & Max Planck Institute for Mathematics in the Sciences,
04103 Leipzig, Germany; Academy\\ & of Mathematics and Systems Science, Chinese Academy of Sciences, Beijing 100190, China.\\
&{Email: shiping@mis.mpg.de}\\
&Mathematics Subject Classification 2010: 31C05, 05C10.
\end{tabular}}

\renewcommand{\thefootnote}{\arabic{footnote}}

\newtheorem{theorem}{Theorem}[section]
\newtheorem{con}[theorem]{Conjecture}
\newtheorem{lemma}[theorem]{Lemma}
\newtheorem{corollary}[theorem]{Corollary}
\newtheorem{definition}[theorem]{Definition}
\newtheorem{conv}[theorem]{Assumption}
\newtheorem{remark}[theorem]{Remark}

\bigskip

\begin{abstract}  We apply Alexandrov geometry
  methods to study geometric analysis aspects of  infinite semiplanar
  graphs with nonnegative combinatorial curvature. We obtain the metric classification of these
  graphs and construct the graphs embedded in the projective plane
  minus one point. Moreover, we show the volume doubling property and
  the Poincar\'e inequality on such graphs. The quadratic volume
  growth of these graphs implies the parabolicity. Finally, we
  prove the polynomial growth harmonic function theorem analogous to
  the case of Riemannian manifolds.
\medskip
\end{abstract}
\section{Introduction}
In this paper, we study systematically (infinite) semiplanar graphs
$G$ of nonnegative curvature. This curvature condition can either be
formulated purely combinatorically, as in the approach of
\cite{St,G,I}, or as an Alexandrov curvature condition on the
polygonal surface $S(G)$ obtained by assigning length one to every
edge and filling in faces. The fact that these two curvature
conditions -- nonnegative combinatorial curvature of $G$ and
nonnegative Alexandrov curvature of $S(G)$ -- are equivalent will be
systematically exploited in the present paper. First of all, we can
then classify such graphs. Curiously, as soon as the maximal degree
of a face is at least 43, the graph necessarily has a rather special
structure. This will simplify our reasoning considerably. Secondly,
as Alexandrov geometry is a natural generalization of Riemannian
geometry, we can systematically carry over the geometric function
theory of nonnegatively curved Riemannian manifolds to the setting
of nonnegatively curved semiplanar graphs. Starting with two basic
inequalities, the volume doubling property and the Poincar\'e
inequality, which hold for such spaces, we obtain the Harnack
inequality for harmonic functions by Moser's iteration scheme. Here,
for defining (sub-, super-)harmonic functions, we use the discrete
Laplace operator of $G$. Our main results then say that a
nonnegatively curved semiplanar graph is parabolic in the sense that
it does not support any nontrivial positive superharmonic function
(equivalently, Brownian motion is recurrent), and that the dimension
of the space of harmonic functions of polynomial growth with
exponent at most $d$ is bounded for any $d$. This is an extension of
the solution by Colding-Minicozzi \cite{CM2} of a conjecture of Yau
\cite{Y3} in Riemannian geometry.

Let us now describe the results in more precise technical terms. The
combinatorial curvature for planar graphs was introduced in Stone
\cite{St,St1}, Gromov \cite{G} and Ishida \cite{I}. In \cite{H},
Higuchi conjectured, as a discrete analog of Myers' theorem in
Riemannian geometry, that any planar graph with positive curvature
everywhere is a finite graph. DeVos and Mohar \cite{DM} solved the
conjecture by proving the Gauss-Bonnet formula for infinite planar
graphs. The combinatorial curvature was studied by many authors
\cite{Woe,CC,Ch,SY,RBK,BP1,BP2,K1,K2,K3}.

In this paper, we are interested in infinite graphs. Let $G$ be an
infinite graph embedded in a 2-manifold $S(G)$ such that each face
is homeomorphic to a closed disk with finite edges as the boundary.
This includes the case of a planar graph, and  we call such a
$G=(V,E,F)$ with its sets of vertices $V$, edges $E$, and faces $F$,
a semiplanar graph. For each vertex $x\in V$, the combinatorial
curvature at $x$ is defined as
$$\Phi(x)=1-\frac{d_x}{2}+\sum_{\sigma\ni x}\frac{1}{\deg(\sigma)},$$
where $d_x$ is the degree of the vertex $x$, $\deg(\sigma)$ is the
degree of the face $\sigma$, and the sum is taken over all faces
incident to $x$ (i.e. ${x\in \sigma}$). The idea of this definition
is to measure the difference of $2\pi$ and the total angle
$\Sigma_x$ at the vertex $x$ on the polygonal surface $S(G)$
equipped with a metric structure  obtained from replacing each face
of G with a regular polygon of side lengths one and gluing them
along the common edges. That is,
$$2\pi \Phi(x)=2\pi-\Sigma_x.$$ Let $\chi(S(G))$ denote the Euler characteristic of the surface $S(G)$. The Gauss-Bonnet formula of $G$ in \cite{DM} reads as
$$\sum_{x\in G}\Phi(x)\leq\chi(S(G)),$$
whenever $\Sigma_{x\in G:\Phi(x)<0}\Phi(x)$ converges. Furthermore,
Chen and Chen \cite{CC} proved that if the absolute total curvature
$\Sigma_{x\in G}|\Phi(x)|$ is finite, then G has only finitely many
vertices with nonvanishing curvature.  Then Chen \cite{Ch} obtained
the topological classification of infinite semiplanar graphs with
nonnegative curvature: $\mathds{R}^2$, the cylinder without
boundary, and the projective plane minus one point. In addition, at
the end of the paper \cite{Ch}, he proposed a question on the
construction of semiplanar graphs with nonnegative curvature
embedded in the projective plane minus one point.

We note that the definition of the combinatorial curvature is
equivalent to the generalized sectional (Gaussian) curvature of the
surface S(G). The semiplanar graph G has nonnegative combinatorial
curvature if and only if the corresponding regular polygonal surface
S(G) is an Alexandrov space with nonnegative sectional curvature,
i.e. $\Sec\ S(G)\geq 0$ (or $\Sec (G)\geq0$ for short).

Here, we are referring to another notion of curvature for such
polygonal spaces, or more precisely, of curvature bounds. This paper
will derive its insights from comparing these curvature notions. A
metric space $(X,d)$ is called an Alexandrov space if it is a
geodesic space (i.e. each pair of points in $X$ can be joined by a
shortest path called a geodesic) and  locally satisfies the
Toponogov triangle comparison. For the basic facts of Alexandrov
spaces, readers are referred to \cite{BGP,BBI}. In this paper, we
shall apply the Alexandrov geometry to study the geometric and
analytic properties of semiplanar graphs with nonnegative curvature.

Alexandrov geometry can be seen as a natural generalization of
Riemannian geometry, and many fundamental results of Riemannian
geometry extend to the more general Alexandrov setting. Firstly, the
well known Cheeger-Gromoll splitting theorem for Riemannian
manifolds with nonnegative Ricci curvature was generalized to
Alexandrov spaces (see \cite{CG,BBI,Mil,M,ZZ}); the result is that
if the $n$-dimensional Alexandrov space $(X,d)$ with nonnegative
curvature contains an infinite geodesic $\gamma$, i.e. $\gamma:
(-\infty, \infty)\rightarrow X$, then $X$ isometrically splits as
$Y\times\mathds{R}$, where $Y$ is an $(n-1)$-dimensional Alexandrov
space with nonnegative curvature. In the present paper, we shall
prove that if the semiplanar graph $G$ with nonnegative curvature
has at least two ends (geometric ends at infinity), then $S(G)$ is
isometric to the cylinder; this is interesting since we do not use
the Gauss-Bonnet formula here. Moreover, we give the metric
classification of $S(G)$ for semiplanar graphs $G$ with nonnegative
curvature. An orientable $S(G)$ is isometric to a plane, or a
cylinder without boundary if it has vanishing curvature everywhere,
and isometric to a cap which is homeomorphic but not isometric to
the plane if it has at least one vertex with positive curvature. A
nonorientable $S(G)$ is isometric to the metric space obtained by
gluing in some way the boundary of $[0,a]\times \mathds{R}$ with
vanishing curvature everywhere (see Lemma \ref{PJC}). By this lemma,
we answer the question of Chen \cite{Ch}.

Secondly, we prove that $G$ inherits some geometric  estimates  from
those of $S(G)$. Let $d^G$ (resp. $d$) denote the intrinsic metric
on the graph $G$ (resp. polygonal surface $S(G)$). It will be proved
that these two metrics are bi-Lipschitz equivalent on $G$, i.e. for
any $x,y \in G$, $$Cd^G(x,y)\leq d(x,y)\leq d^G(x,y).$$  We denote
by $B_R(p)=\{x\in G: d^G(p,x)\leq R\}$ the closed geodesic ball in
$G$ and by $B_R^{S(G)}(p)=\{x\in S(G): d(p,x)\leq R\}$ the closed
geodesic ball in $S(G)$ respectively. The volume of $B_R(p)$ is
defined as $|B_R(p)|=\sum_{x\in B_R(p)}d_x.$ The Bishop-Gromov
volume comparison holds on the $n$-dimensional Alexandrov space
$(X,d)$ with nonnegative curvature (see \cite{BBI}). For any $p\in
X, 0<r<R$, we have
\begin{equation}\label{RVC1}\frac{\mathcal{H}^n(B_R^X(p))}{\mathcal{H}^n(B_r^X(p))}\leq\left(\frac{R}{r}\right)^n,\end{equation}
 \begin{equation}\label{VD1}\mathcal{H}^n(B_{2R}^X(p))\leq2^n \mathcal{H}^n(B_R^X(p)),\end{equation}
 \begin{equation}\label{VG1}\mathcal{H}^n(B_R^X(p))\leq C(n)R^n,\end{equation} where $B_R^X(p)$ is the closed geodesic ball in $X$ and $\mathcal{H}^n$ is the $n$-dimensional Hausdorff measure.  We call (\ref{RVC1}) the relative volume comparison and (\ref{VD1}) the volume doubling property. Note that $S(G)$ is a 2-dimensional Alexandrov space with nonnegative curvature if $G$ is a semiplanar graph with nonnegative combinatorial curvature. Let $D_G$ denote the maximal degree of the faces in $G$, i.e. $D_G=\max_{\sigma\in F}\deg(\sigma)$ which is finite by \cite{CC}. In this paper, for simplicity we also denote $D:=D_G$ when it does not make any confusion. The relative volume growth property for the graph $G$ is obtained in the following theorem.

\begin{theorem} Let G be a semiplanar graph with $\Sec (G)\geq 0$. Then for any $p\in G, 0<r<R,$ we have
 \begin{equation}\label{RVCG1}\frac{|B_R(p)|}{|B_r(p)|}\leq C(D)\left(\frac{R}{r}\right)^2,\end{equation}
 \begin{equation}\label{VDG1}|B_{2R}(p)|\leq C(D) |B_R(p)|,\end{equation}
 \begin{equation}\label{QVG1}|B_R(p)|\leq C(D)R^2,\ \ \ \ (R\geq1)\end{equation} where $C(D)$ is a constant only depending on $D$ which is the maximal facial degree of $G$.
\end{theorem}

Thirdly, we show that the Poincar\'e inequality holds on the
semiplanar graph $G$ with nonnegative curvature. The Poincar\'e
inequality has been proved on Alexandrov spaces in \cite{KMS,Hu1},
and also on graphs ($\epsilon$-nets) embedded into Riemannian
manifolds with bounded geometry in \cite{CS}. Let $u$ be a local
$W^{1,2}$ function on an $n$-dimensional Alexandrov space $(X,d)$
with $\Sec\ X\geq 0$, then
\begin{equation}\label{PI1}\int_{B_R^X(p)}|u-u_{B_R}|^2\leq C(n)R^2\int_{B_R^X(p)}|\triangledown u|^2,\end{equation} where $u_{B_R}=\frac{1}{\mathcal{H}^n(B_R^X(p))}\int_{B_R^X(p)}u.$ For any function $f:G\rightarrow \mathds{R}$, we extend it to each edge of $G$ by  linear interpolation and then to each face nicely with controlled energy (see Lemma \ref{EX1}). So we get a local $W^{1,2}$ function on $S(G)$ which satisfies the Poincar\'e inequality (\ref{PI1}), and then it implies the Poincar\'e inequality on the graph $G$.

\begin{theorem} Let $G$ be a semiplanar graph with $\Sec (G)\geq 0$. Then there exist two constants C(D) and C such that for any $p\in G, R>0, f:B_{CR}(p)\rightarrow \mathds{R}$, we have
\begin{equation}\label{PIG1}\sum_{x\in B_R(p)}(f(x)-f_{B_R})^2d_x\leq C(D)R^2\sum_{x,y\in B_{CR}(p);x\sim y}(f(x)-f(y))^2,\end{equation}
where $f_{B_R}=\frac{1}{|B_R(p)|}\sum_{x\in B_R(p)}f(x)d_x$, and
$x\sim y$ means x and y are neighbors.
\end{theorem}

Finally, we shall study some global properties of harmonic functions
on the semiplanar graph $G$ with nonnegative curvature. Let
$f:G\rightarrow \mathds{R}$ be a function on the graph $G$. The
Laplace operator $L$ is defined as (see \cite{G3,DK,Chu})
$$Lf(x)=\frac{1}{d_x}\sum_{y\sim x}(f(y)-f(x)).$$ A function $f$ is
called harmonic (subharmonic, superharmonic) if $Lf(x)=0\ (\geq
0,\leq 0),$ for each $x\in G$.

A manifold or a graph is called parabolic if it does not admit any
nontrivial positive superharmonic function. The question when a
manifold is parabolic has been studied extensively in the
literature; in fact, parabolicity is equivalent to recurrency for
Brownnian motion (see \cite{G2,HK,RSV}). Noticing that the
semiplanar graph $G$ with nonnegative curvature has the quadratic
volume growth (\ref{QVG1}), we obtain the following theorem in a
standard manner (see \cite{Woe1,HK}).

\begin{theorem} Any semiplanar graph $G$ with $\Sec (G)\geq 0$ is parabolic.
\end{theorem}

Since Yau \cite{Y1} proved the Liouville theorem for positive
harmonic functions on complete Riemannian manifolds with nonnegative
Ricci curvature, the study of harmonic functions on manifolds has
been one of the central fields of geometric analysis. Yau
conjectured in \cite{Y2,Y3} that the linear space of polynomial
growth harmonic functions with a fixed growth rate on a Riemannian
manifold with nonnegative Ricci curvature is of finite dimension.
Colding and Minicozzi \cite{CM2} gave an affirmative answer to the
conjecture by the volume doubling property and the Poincar\'e
inequality. An alternative method by the mean value inequality was
introduced by Colding and Minicozzi \cite{CM4} (see also \cite{L1}).
In this paper, we call this result the polynomial growth harmonic
function theorem. Delmotte \cite{D1} proved it in the graph setting
by assuming the volume doubling property and the Poincar\'e
inequality. Kleiner \cite{Kl} generalized it to Cayley graphs of
groups of polynomial growth, by which he gave a new proof of
Gromov's theorem in group theory. The first author \cite{Hu2}
generalized it to Alexandrov spaces and gave the optimal dimension
estimate analogous to the Riemannian manifold case.

Let $G$ be a semiplanar graph with nonnegative curvature and
$H^d(G)=\{u: Lu\equiv0, |u(x)|\leq C(d^G(p,x)+1)^d\}$ which is the
space of polynomial growth harmonic functions of growth degree less
than or equal to $d$ on $G$. By the method of Colding and Minicozzi
\cite{CM1,CM2,CM3}, the volume doubling property (\ref{VDG1}) and
the Poincar\'e inequality (\ref{PIG1}) imply that $\dim H^d(G)\leq
C(D)d^{v(D)}$ for any $d\geq 1,$ where $C(D)$ and $v(D)$ depend on
$D$ (see \cite{D1}). Instead of the volume doubling property
(\ref{VDG1}), inspired by \cite{CM3}, we use the relative volume
comparison (\ref{RVCG1}) to show that $\dim H^d(G)\leq C(D)d^{2}$.
It seems natural that the dimension estimate of $H^d(G)$ should
involve the maximal facial degree $D$ because the relative volume
comparison and the Poincar\'e inequality cannot avoid $D$, but the
estimate is still not satisfactory since $C(D)$ here is only a
dimensional constant in the Riemannian case.

Furthermore, we note that a semiplanar graph $G$ with nonnegative
curvature and $D_G\geq 43$ has a special structure of linear volume
growth like a one-sided cylinder, see Theorem \ref{LFC}. Inspired by
the work \cite{Sor}, in which Sormani proved that any polynomial
growth harmonic function on a  Riemannian manifold with one end and
nonnegative Ricci curvature of linear volume growth is constant, we
obtain the following theorem.

\begin{theorem} Let $G$ be a semiplanar graph with $\Sec (G)\geq 0$ and $D_G\geq 43$. Then for any $d>0$,
$$\dim H^d(G)=1.$$
\end{theorem}

The final dimension estimate follows from combining the previous two
estimates.

\begin{theorem} Let $G$ be a semiplanar graph with $\Sec (G)\geq 0$. Then for any $d\geq 1$,
$$\dim H^d(G)\leq Cd^2,$$ where $C$ is an absolute constant.
\end{theorem}

For convenience, we may change the values of the constants $C, C(D)$
from line to line in the sequel.

\section{Preliminaries}

A graph is called planar if it can be embedded in the plane without
self-intersection of edges. We define a semiplanar graph similarly.
\begin{definition} A graph $G=(V,E)$ is called semiplanar if it can be embedded into a connected 2-manifold $S$ without self-intersection of edges and each face is homeomorphic to the closed disk with finite edges as the boundary.
\end{definition}

The embedding in the definition is called a strong embedding in
\cite{Ch}. Let $G=(V,E,F)$ denote the semiplanar graph with the set
of vertices, $V$, edges, $E$ and faces, $F$. Edges and faces are
regarded as closed subsets of $S$, and two objects from $V, E, F$
are called incident if one is a proper subset of the other. In this
paper, essentially for simplicity, we shall always assume that the
surface $S$ has no boundary except in Remark \ref{RMB} and $G$ is a
simple graph, i.e. without loops and multi-edges. Throughout this
paper, we write $x\in G$ instead $x\in V$ for the vertex $x.$ We
denote by $d_x$ the degree of the vertex $x\in G$ and by
$\deg(\sigma)$ the degree of the face $\sigma\in F$, i.e. the number
of edges incident to $\sigma$. Further, we assume that $3\leq d_x<
\infty$ and $3\leq \deg(\sigma)<\infty$ for each vertex $x$ and face
$\sigma$, which implies that $G$ is a locally finite graph. For each
semiplanar graph $G=(V,E,F)$, there is a unique metric space,
denoted by $S(G)$, which is obtained from replacing each face of $G$
by a regular polygon of side length one with the same facial degree
and gluing the faces along the common edges in $S$. $S(G)$ is called
the regular polygonal surface of the semiplanar graph $G$.

For a semiplanar graph $G$, the combinatorial curvature at each
vertex $x\in G$ is defined as
$$\Phi(x)=1-\frac{d_x}{2}+\sum_{\sigma\ni x}\frac{1}{\deg(\sigma)},$$ where the sum is taken over all the faces incident to $x$. This curvature can be read from the corresponding regular polygonal surface $S(G)$ as,
$$2\pi\Phi(x)=2\pi-\Sigma_x,$$ where $\Sigma_x$ is the total angle of
$S(G)$ at $x$. Positive curvature thus means  convexity at the
vertex. We shall prove that the semiplanar graph $G$ has nonnegative
curvature everywhere if and only if the regular polygonal surface
$S(G)$ is an Alexandrov space with nonnegative curvature, which is a
generalized sectional (Gaussian) curvature on metric spaces. In this
paper, we denote by $\Sec (G)\geq 0$ the semiplanar graph $G$ with
nonnegative combinatorial curvature and by $\Sec\ X\geq 0$ the
metric space $X$ with nonnegative curvature in the sense of
Alexandrov.

We recall some basic facts in metric geometry and Alexandrov
geometry. Readers are referred to \cite{BGP,BBI}.

A curve $\gamma$ in a metric space $(X,d)$ is a continuous map
$\gamma : [a,b]\rightarrow X.$ The length of a curve $\gamma$ is
defined as
$$L(\gamma)=\sup\left\{\sum_{i=1}^Nd(\gamma(y_{i-1}),\gamma(y_i)):
\text{any\ partition }a=y_0<y_1<\ldots<y_N=b\right\}.$$ A curve
$\gamma$ is called rectifiable if $L(\gamma)<\infty.$ Given $x, y\in
X$, denote by $\Gamma(x,y)$ the set of rectifiable curves joining
$x$ and $y$. A metric space $(X,d)$ is called a length space if
$d(x,y)=\inf_{\gamma\in\Gamma(x,y)}\{L(\gamma)\},$ for any $x,y \in
X$. A curve $\gamma:[a,b]\rightarrow X$ is called a geodesic if
$d(\gamma(a),\gamma(b))=L(\gamma).$ It is always true by the
definition of the length of a curve that $d(\gamma(a),\gamma(b))\leq
L(\gamma).$ A geodesic is a shortest curve (or shortest path)
joining the two end-points. A geodesic space is a length space
$(X,d)$ satisfying that for any $x,y\in X,$ there is a (not
necessarily unique) geodesic joining $x$ and $y$.

Denote by $\Pi_{\kappa},$ $\kappa\in \mathds{R}$ the model space
which is a 2-dimensional, simply connected space form of constant
curvature $\kappa$. Typical ones are $$\Pi_{\kappa}=\left\{
\begin{array}{ll}
\mathds{R}^2,&\kappa=0\\
S^2,&\kappa=1\\
\mathds{H}^2,&\kappa=-1
\end{array}\right..$$ In a geodesic space $(X,d)$, we denote by $\gamma_{xy}$ one of the geodesics joining $x$ and $y$, for $x,y\in X$. Given three points $x, y , z \in X,$ denote by $\triangle_{xyz}$ the geodesic triangle with edges $\gamma_{xy}, \gamma_{yz}, \gamma_{zx}.$ There exists a unique (up to an isometry) geodesic triangle, $\triangle_{\bar x\bar y\bar z}$, in $\Pi_{\kappa}$ ($d(x,y)+d(y,z)+d(z,x)<\frac{2\pi}{\sqrt{\kappa}}$ if ${\kappa>0}$) such that $d(\bar x, \bar y)=d(x,y), d(\bar y, \bar z)=d(y,z)$ and $d(\bar z, \bar x)=d(z,x).$ We call $\triangle_{\bar x\bar y\bar z}$ the comparison triangle in $\Pi_{\kappa}$.

\begin{definition} A complete geodesic space $(X,d)$ is called an Alexandrov space with sectional curvature bounded below by $\kappa$ ($\Sec\ X\geq \kappa$ for short) if for any $p\in X$, there exists a neighborhood $U_p$ of $p$ such that for any $x,y,z\in U_p$ (with $d(x,y)+d(y,z)+d(z,x)<\frac{2\pi}{\sqrt{\kappa}}$ if $\kappa>0$), any geodesic triangle $\triangle_{xyz}$, and any $w\in\gamma_{yz}$, letting $\bar w\in \gamma_{\bar y \bar z}$ be in the comparison triangle $\triangle_{\bar x\bar y\bar z}$ in $\Pi_{\kappa}$ satisfying $d(\bar y,\bar w)=d(y,w)$ and $d(\bar w,\bar z)=d(w,z)$,  we have $$d(x,w)\geq d(\bar x,\bar w).$$
\end{definition}

In other words, an Alexandrov space $(X,d)$ is a geodesic space
which locally satisfies the Toponogov triangle comparison theorem
for the sectional curvature. It is proved in \cite{BGP} that the
Hausdorff dimension of an Alexandrov space $(X,d)$, $\dim_H(X)$, is
an integer or infinity. One dimensional Alexandrov spaces are:
straight line, $S^1$, ray and closed interval.

Let $(X,d)$ be an Alexandrov space, $B_R^X(p)$ denote the closed
geodesic ball centered at $p\in X$ of radius $R>0$, i.e.
$B_R^X(p)=\{x\in X: d(p,x)\leq R\}.$ The well known Bishop-Gromov
volume comparison theorem holds on Alexandrov spaces \cite{BBI}.

\begin{theorem} Let $(X,d)$ be an $n$-dimensional Alexandrov space with nonnegative curvature, i.e. $\Sec\ X\geq0.$ Then for any $p\in X, 0<r<R,$ it holds that
\begin{equation}\label{RVCA1}\frac{\mathcal{H}^n(B_R^X(p))}{\mathcal{H}^n(B_r^X(p))}\leq\left(\frac{R}{r}\right)^n,\end{equation}
\begin{equation}\label{VDA1}\mathcal{H}^n(B_{2R}^X(p))\leq2^n \mathcal{H}^n(B_R^X(p)),\end{equation} \begin{equation}\label{VGA1}\mathcal{H}^n(B_R^X(p))\leq C(n)R^n,\end{equation} where $\mathcal{H}^n$ is the $n$-dimensional Hausdorff measure.
\end{theorem}

A curve $\gamma:(-\infty, \infty)\rightarrow X$ is called an
infinite geodesic if for any $s,t\in(-\infty,\infty)$,
$d(\gamma(s),\gamma(t))=L(\gamma|_{[s,t]}),$ i.e. every restriction
of $\gamma$ to a subinterval is a geodesic (shortest path). For two
metric spaces $(X,d_X), (Y,d_Y),$ the metric product of $X$ and $Y$
is a product space $X\times Y$ equipped with the metric $d_{X\times
Y}$ which is defined as
$$d_{X\times Y}((x_1,y_1),(x_2, y_2))=\sqrt{d_X^2(x_1,x_2)+d_Y^2(y_1,y_2)},$$ for any $(x_1,y_1),(x_2,y_2)\in X\times Y$.
The Cheeger-Gromoll splitting theorem holds on Alexandrov spaces
with nonnegative curvature \cite{BBI,Mil,M,ZZ}.

\begin{theorem}\label{SPLIT} Let $(X,d)$ be an $n$-dimensional
  Alexandrov space with $\Sec\ X\geq0$. If it contains an infinite
  geodesic, then $X$ is isometric to a metric product
  $Y\times\mathds{R},$ where $Y$ is an $(n-1)$-dimensional Alexandrov
  space with $\Sec\ Y\geq0.$
\end{theorem}

Let $(X,d)$ be an $n$-dimensional Alexandrov space with $\Sec\
X\geq\kappa, \kappa\in\mathds{R}.$ The tangent space at each point
$p\in X$ is well defined, denoted by $T_pX$, which is the pointed
Gromov-Hausdorff limit of the rescaling sequence $(X,\lambda d,p)$
as $\lambda\rightarrow\infty$ (see \cite{BBI}). A point $p\in X$ is
called regular (resp. singular) if $T_pX$ is (resp. not) isometric
to $\mathds{R}^n$. Let $S(X)$ denote the set of singular points in
$X$. It is known that $\mathcal{H}^n(S(X))=0$. Otsu and Shioya
\cite{OS} obtained the $C^1$-differentiable and $C^0$-Riemannian
structure on the regular part of $X$, $X\setminus S(X)$. A function
$f$ defined on a domain $\Omega\subset X$ is called Lipschitz if
there is a constant C such that for any $x,y\in \Omega$,
$|f(x)-f(y)|\leq Cd(x,y)$. It can be shown that every Lipschitz
function is differentiable $\mathcal{H}^n$-almost everywhere and
with bounded gradient $|\nabla f|$ (see \cite{C}). Let $Lip(\Omega)$
denote the set of Lipschitz functions on $\Omega$. For any
precompact domain $\Omega\subset X$ and $f\in Lip(\Omega)$, the
$W^{1,2}$ norm of $f$ is  defined as
$$\|f\|_{W^{1,2}(\Omega)}^2=\int_{\Omega}f^2+\int_{\Omega}|\nabla f|^2.$$ The $W^{1,2}$ space on $\Omega$, denoted by $W^{1,2}(\Omega)$, is the completion of $Lip(\Omega)$ with respect to the $W^{1,2}$ norm. A function $f\in W^{1,2}_{loc}(X)$ if for any precompact domain $\Omega\subset\subset X$, $f|_{\Omega}\in W^{1,2}(\Omega)$. The Poincar\'e inequality was proved in \cite{KMS,Hu1}.
\begin{theorem} Let $(X,d)$ be an $n$-dimensional Alexandrov space
  with $\Sec\ X\geq 0$ and $u\in W^{1,2}_{loc}(X),$ then for any
  $R>0$ and $p\in X,$
\begin{equation}\label{PIA1}\int_{B_R^X(p)}|u-u_{B_R}|^2\leq C(n)R^2\int_{B_R^X(p)}|\triangledown u|^2,\end{equation} where $u_{B_R}=\frac{1}{\mathcal{H}^n(B_R^X(p))}\int_{B_R^X(p)}u.$
\end{theorem}

Let $(X,d)$ be a geodesic space and
$\{B_{R_i}^X(p)\}_{i=1}^{\infty}$ be an exhaustion of $X$, i.e.
$B_{R_i}^X(p)\subset B_{R_{i+1}}^X(p)$ for any $i\geq 1$ and
$X=\bigcup_{\i=1}^\infty B_{R_i}^X(p)$, equivalently $R_i\leq
R_{i+1}$ and $R_i\rightarrow\infty$ as $i\rightarrow \infty$. A
connected component $E$ of $X\setminus B_{R_i}^X(p)$ is called
connecting to infinity if there is a sequence of points
$\{q_j\}_{j=1}^{\infty}$ in $E$ such that $d(p,q_j)\rightarrow
\infty$ as $j\rightarrow \infty$. The number of connected components
of $X\setminus B_{R_i}^X(p)$ connecting to infinity, denoted by
$N_i$, is nondecreasing in $i$. Then the limit
$N(X)=\lim_{i\rightarrow\infty}N_i$ is well defined and called the
number of ends of $X$. It is easy to show that $N(X)$ does not
depend on the choice of the exhaustion of $X$,
$\{B_{R_i}^X(p)\}_{i=1}^{\infty}$. Given a connected graph
$G=(V,E)$, let $G_1$ denote the 1-dimensional simplicial complex of
$G$, i.e. a metric space obtained from $G$ by assigning each edge
the length one. Then $G_1$ is a geodesic space and $N(G_1)$ is well
defined. If $G$ is a semiplanar graph and $S(G)$ is the
corresponding regular polygonal surface, then we can also define the
number of ends of $S(G)$, $N(S(G))$.

In the sequel, we recall some facts on the combinatorial structure
of semiplanar graphs. The Gauss-Bonnet formula for the semiplanar
graph was proved in \cite{DM,CC}.

\begin{theorem}\label{GBFT} Let $G$ be a semiplanar graph, $S(G)$ be the corresponding regular polygonal surface, and $t=N(S(G))$. If $G$ has only finitely many vertices with negative curvature, then there exists a closed 2-manifold M, so that $S(G)$ is homeomorphic to $M$ minus t points, and
\begin{equation}\label{GBF}\sum_{x\in G}\Phi(x)\leq \chi(S(G)):=\chi(M)-t.\end{equation}
Moreover, $G$ has at most finitely many vertices with nonvanishing
curvature.
\end{theorem}

By the Gauss-Bonnet formula, Chen \cite{Ch} gave the topological
classification of semiplanar graphs with nonnegative curvature.
\begin{theorem}\label{Chens} Let G be an infinite semiplanar graph with nonnegative curvature everywhere and $S(G)$ be the regular polygonal surface.
Then $S(G)$ is homeomorphic to: $\mathds{R}^2$, the cylinder without
boundary or the projective plane minus one point.
\end{theorem}

Let $G$ be a semiplanar graph and $x\in G$. It is straightforward
that $3\leq d_x\leq 6$ if $\Phi(x)\geq0$ and $3\leq d_x\leq5$ if
$\Phi(x)>0$. A pattern of a vertex $x$ is a vector
$(\deg(\sigma_1),\deg(\sigma_2),\cdots,\deg(\sigma_{d_x})),$  where
$\{\sigma_i\}_{i=1}^{d_x}$ are the faces incident to $x$ ordered
with $\deg(\sigma_1)\leq\deg(\sigma_2)\leq \cdots
\leq\deg(\sigma_{d_x}).$ The following table is the list of all
possible patterns of a vertex $x$ with positive curvature (see
\cite{DM,CC}).

\begin{tabular}{|lc|lr|}
\hline
Patterns&&$\Phi(x)$&\\
    \hline
 $(3,3,k)$ & $3\leq k$&$=1/6+1/k$&\\ $(3,4,k)$  & $4\leq k$  &$=1/12+1/k$&\\ $(3,5,k)$ & $5\leq k$&$=1/30+1/k$&\\
$(3,6,k)$&$6\leq k$&$=1/k$&\\$(3,7,k)$&$7\leq k\leq41$&$\geq1/1722$&\\$(3,8,k)$&$8\leq k\leq 23$&$\geq1/552$&\\
$(3,9,k)$&$9\leq k\leq 17$&$\geq1/306$&\\$(3,10,k)$&$10\leq k\leq 14$&$\geq1/210$&\\$(3,11,k)$&$11\leq k\leq 13$&$\geq1/858$&\\
$(4,4,k)$&$4\leq k$&$=1/k$&\\$(4,5,k)$&$5\leq k\leq 19$&$\geq1/380$&\\$(4,6,k)$&$6\leq k\leq 11$&$\geq1/132$&\\
$(4,7,k)$&$7\leq k\leq 9$&$\geq1/252$&\\$(5,5,k)$&$5\leq k\leq 9$&$\geq1/90$&\\$(5,6,k)$&$6\leq k\leq 7$&$\geq1/105$&\\
$(3,3,3,k)$&$3\leq k$&$=1/k$&\\$(3,3,4,k)$&$4\leq k\leq 11$&$\geq1/132$&\\$(3,3,5,k)$&$5\leq k\leq 7$&$\geq1/105$&\\
$(3,4,4,k)$&$4\leq k\leq 5$&$\geq1/30$&\\$(3,3,3,3,k)$&$3\leq k\leq5$&$\geq1/30$&\\
    \hline
\end{tabular}

\bigskip
All possible patterns of a vertex with vanishing curvature are (see
\cite{GS,CC}):
$$(3,7,42),(3,8,24),(3,9,18),(3,10,15),(3,12,12),(4,5,20),(4,6,12),$$$$(4,8,8),(5,5,10),(6,6,6),
(3,3,4,12),(3,3,6,6),(3,4,4,6),(4,4,4,4),$$\noindent$\ \ \ \ \ \ \
(3,3,3,3,6),(3,3,3,4,4),(3,3,3,3,3,3).$
\medskip

We recall a lemma in \cite{CC}.
\begin{lemma}\label{GFL} Let G be a semiplanar graph G with $\Sec (G)\geq0$ and $\sigma$ be a face of $G$ with $\deg(\sigma)\geq 43.$ Then
$$\sum_{x\in\sigma}\Phi(x)\geq1.$$
\end{lemma}
\begin{proof}For  completeness, we give the proof of the lemma. Since $\Sec (G)\geq0$ and $\deg(\sigma)\geq 43,$ the only possible patterns of the vertices incident to the face $\sigma$ are: $(3,3,k),(3,4,k),(3,5,k),(3,6,k),(4,4,k)$ and $(3,3,3,k),$ where $k=\deg(\sigma).$ In each case, we have $\Phi(x)\geq\frac{1}{k},$ for $x\in\sigma.$ Hence, we get
$$\sum_{x\in\sigma}\Phi(x)\geq1.$$
\end{proof}

Let $G=(V,E,F)$ be a semiplanar graph. We denote by
$D_G=\sup\{\deg(\sigma):\sigma \in F\}$ the maximal degree of faces
in $G$. If $G$ has nonnegative curvature everywhere, then by Theorem
\ref{GBFT}, $G$ has at most finitely many vertices with nonvanishing
curvature which implies that $D_G<\infty$.

\begin{lemma}\label{GFL1} Let $G$ be an infinite semiplanar graph with $\Sec (G)\geq0.$ Then either $D_G\leq42$, or $G$ has a unique face $\sigma$ with $\deg(\sigma)\geq43$ and has vanishing curvature elsewhere.
\end{lemma}
\begin{proof} If $G$ has a face $\sigma$ with $\deg(\sigma)\geq43,$ then by Lemma \ref{GFL} $$\sum_{x\in\sigma}\Phi(x)\geq1.$$ Since $G$ is an infinite graph with $\Sec (G)\geq0,$ by the Gauss-Bonnet formula (\ref{GBF}), we have
$$\sum_{x\in G}\Phi(x)\leq1,$$ because $\chi(M)\leq2$ and $t\geq1.$ Hence $\sum_{x\in\sigma}\Phi(x)=1$ and $\Phi(y)=0$ for any $y\not\in\sigma.$ Furthermore, the only possible patterns of the vertices incident to $\sigma$ are: $(3,6,k),(4,4,k),(3,3,3,k),$ because the other three patterns $(3,3,k),(3,4,k),$ $(3,5,k)$ have curvature strictly larger than $\frac{1}{k},$ where $k=\deg(\sigma).$
\end{proof}

Let $\mathcal{G}$ denote the set of semiplanar graphs. We define a
graph operation on $\mathcal{G}$,
$P:\mathcal{G}\rightarrow\mathcal{G}.$ For any $G\in \mathcal{G},$
we choose a (possibly infinite) subcollection of hexagonal faces of
G, add new vertices at the barycenters of the hexagons, and join
them to the vertices of the hexagons by new edges. In such a way, we
obtain a new semiplanar graph, denoted by $P(G)$, which replaces
each hexagon chosen in $G$ by six triangles. We note that
$P:\mathcal{G}\rightarrow\mathcal{G}$ is a multivalued map depending
on which subcollection of hexagons we chosen. The inverse map of
$P$, denoted by $P^{-1}:\mathcal{G}\rightarrow\mathcal{G},$ is
defined as a semiplanar graph $P^{-1}(G)$ obtained from replacing
couples of six triangles incident to a common vertex of pattern
$(3,3,3,3,3,3)$ in $G$ by a hexagon (we require that the hexagons do
not overlap). It is easy to see that $S(P(G))$ and $S(P^{-1}(G))$
are isometric to $S(G)$ which implies that the graph operations $P$
and $P^{-1}$ preserve the curvature condition, i.e. $\Sec\
S(P(G))\geq0$ (or $\Sec\ S(P^{-1}(G))\geq0$) $\Longleftrightarrow$
$\Sec\ S(G)\geq0$.

We investigate the combinatorial structure of the semiplanar graph
$G$ with nonnegative curvature and large face degree, i.e. $D_G\geq
43.$ Lemma \ref{GFL1} shows that there is a unique large face
$\sigma$ such that $\deg(\sigma)=D_G=k\geq43$ and the only patterns
of vertices of $\sigma$ are: $(3,6,k),(4,4,k)$ and $(3,3,3,k).$
Without loss of generality, by the graph operation $P$, it suffices
to assume that the semiplanar graph $G$ has no hexagonal faces. It
is easy to show that if one of the vertices of $\sigma$ is of
pattern $(4,4,k)$ (or $(3,3,3,k)$), the other vertices incident to
$\sigma$ are of the same pattern. We denote by $L_1$ the set of
faces attached to the large face $\sigma$, which are of the same
type (triangle or square) and for which the boundary of $\sigma\cup
L_1$ has the same number of edges as the boundary of $\sigma.$ By
Lemma \ref{GFL1}, $G$ has vanishing curvature except at the vertices
incident to $\sigma$. Hence, $\sigma\cup L_1$ is in the same
situation as $\sigma$. To continue the process, we denote by $L_2$
the set of faces attached to $\sigma\cup L_1$ which are of the same
type (triangle or square). In this way, we obtain an infinite
sequence of sets of faces, $\sigma, L_1, L_2,\cdots, L_m,\cdots,$
where $L_m$ are the sets of faces of the same type (triangle or
square) for $m\geq1$. $L_m$ and $L_n$ ($m\not=n$) may be different
since they are independent.

\begin{theorem}\label{LFC} Let $G$ be a semiplanar graph with $\Sec (G)\geq0$ and $D_G\geq43,$ and let $\sigma$ be the face of maximal
  degree. Then either $G$ has no hexagons, constructed from a sequence of sets of
  faces, $\sigma, L_1, L_2,\cdots, L_m,\cdots,$ where $L_m$ are the
  sets of faces of the same type (triangle or square), denoted by $S(G)=\sigma\cup\bigcup_{m=1}^{\infty}L_m$, or $G$ has hexagons, i.e. $G=P^{-1}(G')$ where $G'$ has no hexagons.
\end{theorem}

\section{Metric Classification of Semiplanar Graphs with Nonnegative Curvature}
In this section, we prove that any regular polygonal surface is a
complete geodesic space and the combinatorial curvature definition
is consistent with the sectional curvature in the sense of
Alexandrov. We then obtain the metric classification of semiplanar
graphs with nonnegative curvature.

Let $G$ be a semiplanar graph and $S(G)$ be the corresponding
regular polygonal surface. Denote by $G_1$ the $1$-dimensional
simplicial complex with the metric, denoted by $d^{G_1}$, by
assigning each edge the length one. As a subset of $S(G)$, $G_1$ has
another metric, denoted by $d$, which is the restriction of the
intrinsic metric $d$ of $S(G)$ to $G_1$. The following lemma says
that they are bi-Lipschitz equivalent. We note that
$d^G(x,y)=d^{G_1}(x,y),$ for any $x,y\in G.$

\begin{lemma}\label{LEM1} Let $G$ be a semiplanar graph and $S(G)$ be the regular polygonal surface of $G$. Then there exists a constant $C$ such that for any $x,y\in G_1,$
\begin{equation}\label{LEM2}Cd^{G_1}(x,y)\leq d(x,y)\leq d^{G_1}(x,y).\end{equation}
\end{lemma}

To prove the lemma, we need the following lemma in  Euclidean
geometry.

\begin{lemma}\label{EDC1} Let $\triangle_n\subset \mathds{R}^2$ be a regular $n$-polygon of side length one ($n\geq3$). A straight line $L$ intersects the boundary of $\triangle_n$ at two points, $A$ and $B$. Denote by $|AB|=d$ the length of the segment $AB$, by $l_1,l_2$ the length of the two paths $P_1, P_2$ on the boundary of $\triangle_n$ joining $A$ and $B$. Then we have
\begin{equation}\label{EDC2}C\min\{l_1,l_2\}\leq d\leq \min\{l_1,l_2\},\end{equation} where the constant $C$ does not depend on $n.$
\end{lemma}

\begin{proof} It suffices to prove that $d\geq
  C\min\{l_1,l_2\}$. Without loss of generality, we may assume
  $l_1\leq l_2.$ It is easy to prove the lemma for $n=3,$ so we consider $n\geq 4.$ If the shorter path $P_1$ contains no full edges of
  $\triangle_n$, i.e. $A$ and $B$ are on adjacent edges, then $P_1$
  and $AB$ form a triangle. Denote by $a, b$ the lengths of the two sides in $P_1$ and by $\alpha$ the angle opposite to $AB$. Then we have $\alpha=\frac{(n-2)\pi}{n}$ and $l_1=a+b.$ By the cosine rule, we obtain that
$$d\geq a-b\cos\alpha,$$ $$d\geq b-a\cos\alpha.$$ Then it follows that
$$2d\geq (a+b)(1-\cos\alpha)\geq(a+b)(1-\cos\frac{\pi}{3})=\frac{1}{2}l_1.$$ Hence \begin{equation}\label{EDL1}d\geq\frac{1}{4}l_1.\end{equation}

If $P_1$ contains at least one full edge, we consider the following
cases.

\textit{Case 1.} $n\leq 6.$

We choose one full edge in $P_1$ and extend it to a straight line,
then project the path $P_1$ onto the line. It is easy to show that
$$d\geq|Proj P_1|\geq 1,$$ where $Proj P_1$ is the projection of the path $P_1$. Since $n\leq 6,$ we have $l_1\leq3$ and \begin{equation}\label{EDL2}d\geq 1\geq\frac{l_1}{3}.\end{equation}

\textit{Case 2.} $n\geq 7.$

Denote by $l$ the number of full edges contained in $P_1$. We draw
the circumscribed circle of $\triangle_n$, denoted by $C_n$, with
center $O$ of radius $R_n,$ where $2R_n\sin\frac{\pi}{n}=1.$ Let the
straight line $L$ (passing through $A$ and $B$) intersect the circle
$C_n$ at $C$ and $D$ ($C$ is close to $A$). Denote by $d'$ the
length of the segment $CD$, by $\theta$ the angle of $\measuredangle
COD$ and by $l'$ the length of the arc $\wideparen{CD}.$

\textit{Case 2.1.} $l\geq3.$

On one hand, by $l\geq3,$ we have $\theta\geq
l\frac{2\pi}{n}\geq3\frac{2\pi}{n}.$ Hence,
$$d'=2R_n\sin\frac{\theta}{2}=\frac{\sin\frac{\theta}{2}}{\sin\frac{\pi}{n}}\geq\frac{\sin3\frac{\pi}{n}}{\sin\frac{\pi}{n}}=3-4\sin^2\frac{\pi}{n}\geq
3-4\sin^2\frac{\pi}{7}\geq2.24.$$ On the other hand, by $|AC|\leq1$
and $|BD|\leq1$, we obtain that
$$d'-d\leq|AC|+|BD|\leq 2.$$ Then we have
\begin{equation}\label{EDL11}\frac{d}{d'}\geq1-\frac{2}{d'}\geq1-\frac{2}{2.24}=C.\end{equation}
Since $d'=2R_n\sin\frac{\theta}{2}$ and $l'=R_n\theta,$ we have
\begin{equation}\label{EDL12}\frac{d'}{l'}=\frac{2\sin\frac{\theta}{2}}{\theta}\geq\frac{2\cdot\frac{2}{\pi}\cdot\frac{\theta}{2}}{\theta}=\frac{2}{\pi}.\end{equation}

In addition, \begin{equation}\label{EDL13}l'\geq l\geq
l_1-2\geq\frac{l_1}{3},\end{equation} where the last inequality
follows from $l_1\geq l\geq3.$

Hence, by (\ref{EDL11}), (\ref{EDL12}) and (\ref{EDL13}), we have

\begin{equation}\label{EDL3}d\geq Cl_1.\end{equation}

\textit{Case 2.2.} $l=1.$

We denote by $EF$ the full edge contained in $P_1$ ($E$ is close to
$A$) and extend $AE$ and $BF$ to intersect at the point $H$. It is
easy to calculate the angle $\measuredangle
EHF=\pi-\frac{4\pi}{n}\geq\pi-\frac{4\pi}{7}.$ By an argument
similar to the beginning of the proof, we obtain that
\begin{equation}\label{EDL4}d=|AB|\geq\frac{1}{2}(|AH|+|BH|)(1-\cos(\pi-\frac{4\pi}{7}))\geq C(|AE|+|EF|+|FB|)=Cl_1,\end{equation} where the last inequality follows from the triangle inequality.

\textit{Case 2.3.} $l=2.$

We denote by $EF$ and $FH$ the full edges contained in $P_1$ ($E$ is
close to $A$) and extend $AE$ and $BH$ to intersect at the point
$K$. Easy calculation shows that $\measuredangle
EKH=\pi-\frac{6\pi}{n}\geq\pi-\frac{6\pi}{7}.$ By the same argument,
we get
\begin{equation}\label{EDL5}d=|AB|\geq\frac{1}{2}(|AK|+|BK|)(1-\cos(\pi-\frac{6\pi}{7}))\geq Cl_1.\end{equation}

Hence, by (\ref{EDL1}), (\ref{EDL2}), (\ref{EDL3}), (\ref{EDL4}) and
(\ref{EDL5}), we obtain that $$d\geq Cl_1,$$ where $C$ is an
absolute constant. Then the lemma follows.
\end{proof}

Now we prove Lemma \ref{LEM1}.
\begin{proof}[Proof of Lemma \ref{LEM1}] For any $x, y\in G_1,$ it is obvious that $d(x,y)\leq d^{G_1}(x,y).$ Hence it suffices to show the inequality in the opposite direction.
Let $\gamma:[a,b]\rightarrow S(G)$ be a geodesic joining $x$ and
$y$. By the local finiteness assumption of the graph $G$, there
exist finitely many faces that cover the geodesic $\gamma.$ There is
a partition of $[a,b]$, $\{y_i\}_{i=0}^N,$ where
$a=y_0<y_1<\cdots<y_N=b,$ such that $\gamma|_{[y_{i-1},y_{i}]}$ is a
segment on the face $\sigma_i$ and $\gamma(y_{i-1}), \gamma(y_i)$
are on the boundary of $\sigma_i,$ for $1\leq i\leq N.$  For each
$1\leq i\leq N$, we choose the shorter path, denoted by $l_i$, on
the boundary of the face $\sigma_i$ which joins $\gamma(y_{i-1})$
and $\gamma(y_i).$ By Lemma \ref{EDC1}, we get
$$CL(l_i)\leq d(\gamma(y_{i-1}), \gamma(y_i))\leq L(l_i),$$ where $L(l_i)$ is the length of $l_i.$ Connecting $l_i$, we obtain a path $l$ in $G_1$ joining $x$ and $y$. Then we have
$$L(l)=\sum_{i=1}^NL(l_i)\leq\frac{1}{C}\sum_{i=1}^Nd(\gamma(y_{i-1}),\gamma(y_{i}))=\frac{1}{C}d(x,y).$$
Hence, $$d^{G_1}(x,y)\leq L(l)\leq\frac{1}{C}d(x,y).$$

\end{proof}

\begin{theorem}\label{CMS1} Let $G=(V,E,F)$ be a semiplanar graph and $S(G)$ be the regular polygonal surface. Then $(S(G),d)$ is a complete metric space.
\end{theorem}

\begin{proof}
We denote by $S(G)=\bigcup_{\sigma\in F}\sigma$ the regular
polygonal surface of $G$, by $\overline{S(G)}$ the completion of
$S(G)$ with respect to the metric $d$. Let $(\sigma)_{\epsilon_0}$
denote the $\epsilon_0$-neighborhood of $\sigma$ in
$\overline{S(G)}$, for $\epsilon_0>0$. To prove the theorem, it
suffices to show that there exists a constant $\epsilon_0$ such that
for any face $\sigma\in F$ we have $(\sigma)_{\epsilon_0}\subset
S(G).$

For any $\sigma\in F$, let $Q=\bigcup\{\tau\in F:
\tau\cap\sigma\neq\emptyset\}.$ By the local finiteness of $G$, $Q$
is a union of finitely many faces and the boundary of $Q$, $\partial
Q$, has finitely many edges. It is easy to see that
$d^{G_1}(\partial Q, \partial \sigma)=\inf\{d^{G_1}(x,y):
x\in\partial Q, y \in\partial \sigma\}\geq 1.$ By Lemma \ref{LEM1},
we obtain that for any $x\in\partial Q, y \in\partial \sigma,$
$$d(x,y)\geq C=2\epsilon_0,$$ where we choose $\epsilon_0=\frac{C}{2}.$ Then we have $$d(\overline{S(G)}\setminus Q,\sigma)=\inf\left\{d(x,y):x\in \overline{S(G)}\setminus Q, y\in \sigma\right\}\geq2\epsilon_0>\epsilon_0.$$ Hence, it follows that
$$(\sigma)_{\epsilon_0}\subset Q\subset S(G).$$

\end{proof}

\begin{corollary} Let $G$ be a semiplanar graph and $S(G)$ be the regular polygonal surface. Then $G$ has nonnegative (resp. nonpositive) curvature everywhere if and only if $S(G)$ is an Alexandrov space with nonnegative (resp. nonpositive) curvature.
\end{corollary}
\begin{proof} We prove only the case for nonnegative curvature. The proof for the case of nonpositive curvature is similar.

By Theorem \ref{CMS1}, $S(G)$ is a complete metric space. It is
obvious that $S(G)$ is a geodesic space. Suppose $G$ has nonnegative
curvature everywhere. At each point except the vertices, there is a
neighborhood which is isometric to the flat disk in $\mathds{R}^2$.
At the vertex $x\in G$, the curvature condition $\Phi(x)\geq0$ is
equivalent to $\Sigma_x\leq2\pi.$ Then there is a neighborhood of
$x$ (isometric to a conic surface in $\mathds{R}^3$) satisfying the
Toponogov triangle comparison with respect to the model space
$\mathds{R}^2$. Hence, $S(G)$ is an Alexandrov space with $\Sec\
S(G)\geq0.$ Conversely, if $S(G)$ is an Alexandrov space with $\Sec\
S(G)\geq0,$ then the total angle of each point of $S(G)$ is at most
$2\pi$, which implies the nonnegative curvature condition at the
vertices.
\end{proof}

In the following, we investigate the metric structure of regular
polygonal surfaces by Alexandrov space methods, where we don't use
the Gauss-Bonnet formula.

\begin{lemma}\label{ENDL} Let $G=(V,E,F)$ be a semiplanar graph, $G_1$ be the $1$-dimensional simplicial complex and $S(G)$ be the regular polygonal surface. Then we have
$$N(G_1)=N(S(G)).$$
\end{lemma}
\begin{proof} It is easy to show that $N(S(G))\leq N(G_1),$ since $G_1\subset S(G).$ So it suffices to prove that $N(G_1)\leq N(S(G)).$

Let $\{B_{R_i}^{G_1}(p)\}_{i=1}^{\infty}$ be an exhaustion of $G_1,$
such that $G_1\setminus B_{R_i}^{G_1}(p)$ has $N_i$ different
connected components connecting to infinity, denoted by
$E_1^{i},\cdots,E_{N_i}^i,$ and
$N(G_1)=\lim_{i\rightarrow\infty}N_i.$ By the local finiteness of
$G$, $N_i<\infty.$ For any $i\geq1,$ let $Q_i=\bigcup\{\sigma\in
F:\sigma\cap B_{R_i}^{G_1}(p)\neq\emptyset\},$ i.e. the union of the
faces attached to $B_{R_i}^{G_1}(p).$ By the local finiteness of
$G$, $Q_i$ is compact. We shall prove that $S(G)\setminus Q_i$ has
at least $N_i$ different connected components connecting to
infinity, then we have $N(S(G))\geq N_i$ for any $i\geq1,$ which
implies the lemma.

For fixed $i\geq1,$ let $H_j:=E_j^i\cap(S(G)\setminus Q_i),
j=1,\cdots,N_i.$ It is easy to see that $H_j\neq\emptyset,$ since
$E_j^i$ is connecting to infinity for $1\leq j\leq N_i.$ We shall
prove that for any $j\neq k, H_j$ and $H_k$ are disconnected in
$S(G)\setminus Q_i.$ Suppose it is not true, then there exist $x\in
H_j,$ $y\in H_k$ and a curve $\gamma:[a,b]\rightarrow S(G)$ in
$S(G)\setminus Q_i$ joining $x$ and $y,$ i.e.
\begin{equation}\label{ENDL1}\gamma\cap Q_i=\emptyset.\end{equation}
As in the proof of Lemma \ref{LEM1}, we can find a curve
$\gamma':[a,b]\rightarrow G_1$ in $G_1$ such that $\gamma'$ and
$\gamma$ pass through the same faces, i.e. for any $t\in[a,b],$
there is a face $\tau$ such that $\gamma(t)\in \tau$ and
$\gamma'(t)\in\tau.$ Since $H_j$ and $H_k$ are disconnected in
$G_1\setminus B_{R_i}^{G_1}(p),$ we have $\gamma'(t_0)\in
B_{R_i}^{G_1}(p),$ for some $t_0\in[a,b].$ Then there exists a face
$\tau$ such that $\gamma(t_0)\in\tau$ and $\gamma'(t_0)\in\tau.$
Hence $\tau\subset Q_i$ and $\gamma\cap Q_i\neq\emptyset,$ which
contradicts to (\ref{ENDL1}).

\end{proof}

By this lemma, we can apply the Cheeger-Gromoll splitting theorem to
the polygonal surface of the semiplanar graph with nonnegative
curvature.

\begin{theorem}\label{SPLTT} Let $G$ be a semiplanar graph with $\Sec (G)\geq0,$ $S(G)$ be the regular polygonal surface. If $N(G_1)\geq2,$ then $S(G)$ is isometric to a cylinder without boundary.
\end{theorem}
\begin{proof} By Lemma \ref{ENDL}, it follows from $N(G_1)\geq2$ that $N(S(G))\geq2.$ A standard Riemannian geometry argument proves the existence of an infinite geodesic $\gamma:(-\infty,\infty)\rightarrow S(G).$ Since $S(G)$ is an Alexandrov space with nonnegative curvature, the Cheeger-Gromoll splitting theorem, Theorem \ref{SPLIT}, shows that $S(G)$ is isometric to $Y\times \mathds{R},$ where $Y$ is a $1$-dimensional Alexandrov space without boundary, i.e. straight line or circle. Because $N(S(G))\geq2,$ $Y$ must be a circle. Hence, $S(G)$ is isometric to a cylinder without boundary.
\end{proof}

\begin{remark}\label{RMB} \emph{Since the Cheeger-Gromoll splitting theorem holds for Alexandrov space with boundary, we may formulate the above theorem in the case of regular polygonal surfaces with boundary (homeomorphic to a manifold with boundary). For the vertex} $x$ \emph{on the boundary, we define the combinatorial curvature as}
$$\Phi(x)=1-\frac{d_x}{2}+\sum_{\sigma\ni
  x}\frac{1}{\deg(\sigma)}=\frac{\pi-\Sigma_x}{2\pi},$$ \emph{where}
$\Sigma_x$ \emph{is the total angle at} $x.$ \emph{Let} $G$ \emph{be
a
  semiplanar graph with nonnegative curvature everywhere and} $N(G_1)\geq2,$ \emph{then the polygonal surface} $S(G)$ \emph{is isometric to either the cylinder without boundary or the cylinder with boundary, i.e.} $[a,b]\times \mathds{R}.$
\end{remark}

Next we consider the tilings (or tessellations) of the plane (see
\cite{GS}) and the construction of semiplanar graphs with
nonnegative curvature.

Let $G$ be a semiplanar graph with nonnegative curvature and $S(G)$
be the regular polygonal surface of $G.$ If $S(G)$ is isometric to
the plane, $\mathds{R}^2$, then $G$ is just a tiling of the plane by
regular polygons called a regular tiling. Then $G$ has vanishing
curvature everywhere. There are infinitely many tilings of the
plane. A classification is possible only for regular ones. In this
paper, we only consider regular tilings. A tiling is called
monohedral if all tiles are congruent. The only three monohedral
tilings are by triangles,  squares or hexagons. There are 11
distinct tilings such that all vertices are of the same pattern:
$$(3^6)(3^4,6)(3^3,4^2)(3^2,4,3,4)(3,4,6,4)(3,6,3,6)(3,12^2)(4^4)(4,6,12)(4,8^2)(6^3).$$
They are called Archimedean tilings and they clearly include the
three monohedral tilings.

If $S(G)$ has at least two ends, then by Theorem \ref{SPLTT} it is
isometric to a cylinder without boundary and $G$ has vanishing
curvature everywhere. If $S(G)$ is nonorientable, then by Theorem
\ref{Chens} and the Gauss-Bonnet formula (\ref{GBF}) $S(G)$ is
homeomorphic to the projective plane minus one point and $G$ has
vanishing curvature everywhere.

Conversely, if $G$ has vanishing curvature everywhere, then so does
$S(G)$. Hence, $S(G)$ is isometric to $\mathds{R}^2,$ or a cylinder
if it is orientable. $S(G)$ is homeomorphic to the projective plane
minus one point if it is nonorientable.

In addition, if $G$ has positive curvature somewhere, then so does
$S(G),$ which implies that $S(G)$ is not isometric to
$\mathds{R}^2,$ but by the Gauss-Bonnet formula (\ref{GBF}), it is
homeomorphic to $\mathds{R}^2$. We call it a cap.

An isometry of $\mathds{R}^2$ is a mapping of $\mathds{R}^2$ onto
itself which preserves the Euclidean distance. All isometries of
$\mathds{R}^2$ form a group. It is well known that every isometry of
$\mathds{R}^2$ is of one of four types: 1. rotation, 2. translation,
3. reflection in a given line, 4. glide reflection, i.e. a
reflection in a given line composed with a translation parallel to
the same line (see \cite{GS}).

For any planar tiling $\Sigma,$ an isometry is called a symmetry of
$\Sigma$ if it maps every tile of $\Sigma$ onto a tile of $\Sigma.$
It is easy to see that all symmetries of $\Sigma$ form a subgroup of
isometries of $\mathds{R}^2.$ We denote by $S(\Sigma)$ the group of
symmetries of $\Sigma.$ For any $\iota\in S(\Sigma),$ we denote by
$<\iota>$ the subgroup of $S(\Sigma)$ generated by the symmetry
$\iota.$ The metric quotient of $\mathds{R}^2$ by $<\iota>,$ denoted
by $\mathds{R}^2/<\iota>,$ is a metric space with quotient metric
obtained by the group action $<\iota>$ (see \cite{BBI}). The
following lemma shows the construction of the tilings of a cylinder.

\begin{lemma}\label{CYLC} There is a correspondence between a planar tiling $\Sigma$ with a translation symmetry $T,$ $(\Sigma, T)$ and a tiling of a cylinder.
\end{lemma}
\begin{proof} For any planar tiling $\Sigma$ with a translation symmetry $T,$ the metric quotient $\mathds{R}^2/<T>$ is isometric to a cylinder. The tiling $\Sigma$ induces a tiling of $\mathds{R}^2/<T>.$

Conversely, given a tiling $\Sigma'$ of a cylinder $W,$ we lift $W$
to its universal cover $\mathds{R}^2$ by a map
$\pi:\mathds{R}^2\rightarrow W.$ It is easy to see that $\pi$ is
locally isometric, since $W$ is flat. The tiling $\Sigma'$ can be
lifted by $\pi$ to a tiling $\Sigma$ of $\mathds{R}^2,$ which has a
translation symmetry by  construction.
\end{proof}

Next we consider the metric structure of the semiplanar graph with
nonnegative curvature such that the corresponding regular polygonal
surface is nonorientable, i.e. homeomorphic to the projective plane
minus one point.

\begin{lemma}\label{PJC} There is a correspondence between a planar tiling $\Sigma$ with a glide reflection symmetry $\iota,$ $(\Sigma,\iota)$ and a tiling of the projective plane minus one point with nonnegative curvature.
\end{lemma}
\begin{proof} Let $\Sigma$ be a planar tiling with symmetry of a glide reflection $$\iota=T_{a,L}\circ F_L=F_L \circ T_{a,L},$$
where $a>0,$ $L$ is a straight line, $T_{a,L}$ is a translation
along $L$ through distance $a$ and $F_L$ is a reflection in the line
$L.$ The metric quotient $\mathds{R}^2/<\iota>$ is isometric to the
metric space obtained from gluing the boundary of $[0,a]\times
\mathds{R},$ which is perpendicular to the line $L,$ by the glide
reflection $\iota.$ It is easy to see that $\mathds{R}^2/<\iota>$ is
homeomorphic to the projective plane minus one point and has
vanishing curvature everywhere. Hence the planar tiling $\Sigma$ and
the symmetry $\iota$ of $\Sigma$ induce a tiling of
$\mathds{R}^2/<\iota>.$

Conversely, let $\Sigma'$ be a tiling of $\mathds{R}P^2\setminus
\{\underline{o}\},$ with nonnegative curvature (actually with
vanishing curvature everywhere). We construct a covering map of
$\mathds{R}P^2\setminus \{\underline{o}\}$ with a $\mathds{Z}_2$
action,
$$\pi:S^2\setminus \{S,N\}\rightarrow \mathds{R}P^2\setminus
\{\underline{o}\},$$ where $S$ and $N$ are the south and north pole
of $S^2.$ We lift the tiling $\Sigma'$ to a tiling $\Sigma''$ of
$S^2\setminus \{S,N\}.$ Since $\Sigma'$ has vanishing curvature
everywhere, so does the lifted tiling $\Sigma''.$ Note that
$S^2\setminus \{S,N\}$ has two ends. By Theorem \ref{SPLTT}, the
regular polygonal surface $S(\Sigma'')$ is isometric to a cylinder,
denoted by $(\frac{a}{\pi}S^1)\times \mathds{R}$. By Lemma
\ref{CYLC}, the tiling of a cylinder $\Sigma''$ can be regarded to
be a tiling of the cylinder which induces a planar tiling
$\Sigma'''$ and a translation symmetry $T_{2a}$ with
$T_{2a}$-invariant domain $[0,2a]\times\mathds{R}\subset
\mathds{R}^2.$ Since the $\mathds{Z}_2$ action of ${\pi}$, the
tiling $\Sigma'''$ has a glide reflection symmetry $$\iota=F_L \circ
T_{a,L}$$ where $L$ is parallel to the direction of the translation
$T_{2a}.$
\end{proof}

By the discussion above, we obtain the metric classification of
$S(G)$ for a semiplanar graph $G$ with nonnegative curvature.

\begin{theorem}\label{METRICC}Let $G$ be a semiplanar graph with nonnegative curvature and $S(G)$ be the regular polygonal surface of $G.$ If $G$ has positive curvature somewhere, then $S(G)$ is isometric to a cap which is homeomorphic but not isometric to the plane. If $G$ has vanishing curvature everywhere, then $S(G)$ is isometric to a plane, or a cylinder without boundary if it is orientable, and $S(G)$ is isometric to a metric space obtained from gluing the boundary of $[0,a]\times \mathds{R}$ by a glide reflection, $\iota=T_{a,L}\circ F_L,$ where $L$ is perpendicular to the cylinder, if it is nonorientable.
\end{theorem}

At the end of the paper \cite{Ch}, Chen raised a question on the
classification of infinite graphs with nonnegative curvature
everywhere which can be embedded into the projective plane minus one
point. By Lemma \ref{PJC}, it suffices to find the planar tiling
with a glide reflection symmetry.

\begin{theorem} The monohedral tilings of the projective plane minus one point with nonnegative curvature are of three types: triangle, square, hexagon.
\end{theorem}
\begin{proof} By Lemma \ref{PJC}, the monohedral tiling of the projective plane minus one point with nonnegative curvature is induced by the monohedral tiling of
the plane of triangles, of squares or of hexagons and a glide
reflection for the tiling.
\end{proof}

Chen \cite{Ch} gave two classes of monohedral tilings of the
projective plane with nonnegative curvature: $PS_n$ ($n$ is even)
and $PH_n$ ($n$ is odd). $PS_n$ is induced by the monohedral tiling
of the plane of squares. In fact, $PH_n$ ($n$ is odd) is a proper
subset of monohedral tilings of the projective plane minus one point
which are induced by the monohedral tiling of the plane by hexagons.
We give an example below (see Figure \ref{3}, \ref{2}) which is
induced by the tiling of the plane by hexagons, but is not included
in $PH_n$ ($n$ is odd). Let $PT,$ $PS,$ $PH$ denote the tilings of
the projective plane minus one point which are induced by the
monohedral tiling of the plane of triangles, squares, hexagons and a
glide reflection symmetry. They provide the complete classification
of monohedral tilings of the projective plane minus one point with
nonnegative curvature.

\begin{figure}[h]
\begin{minipage}[t]{0.45\linewidth}
\centering
\includegraphics[width=0.65\textwidth]{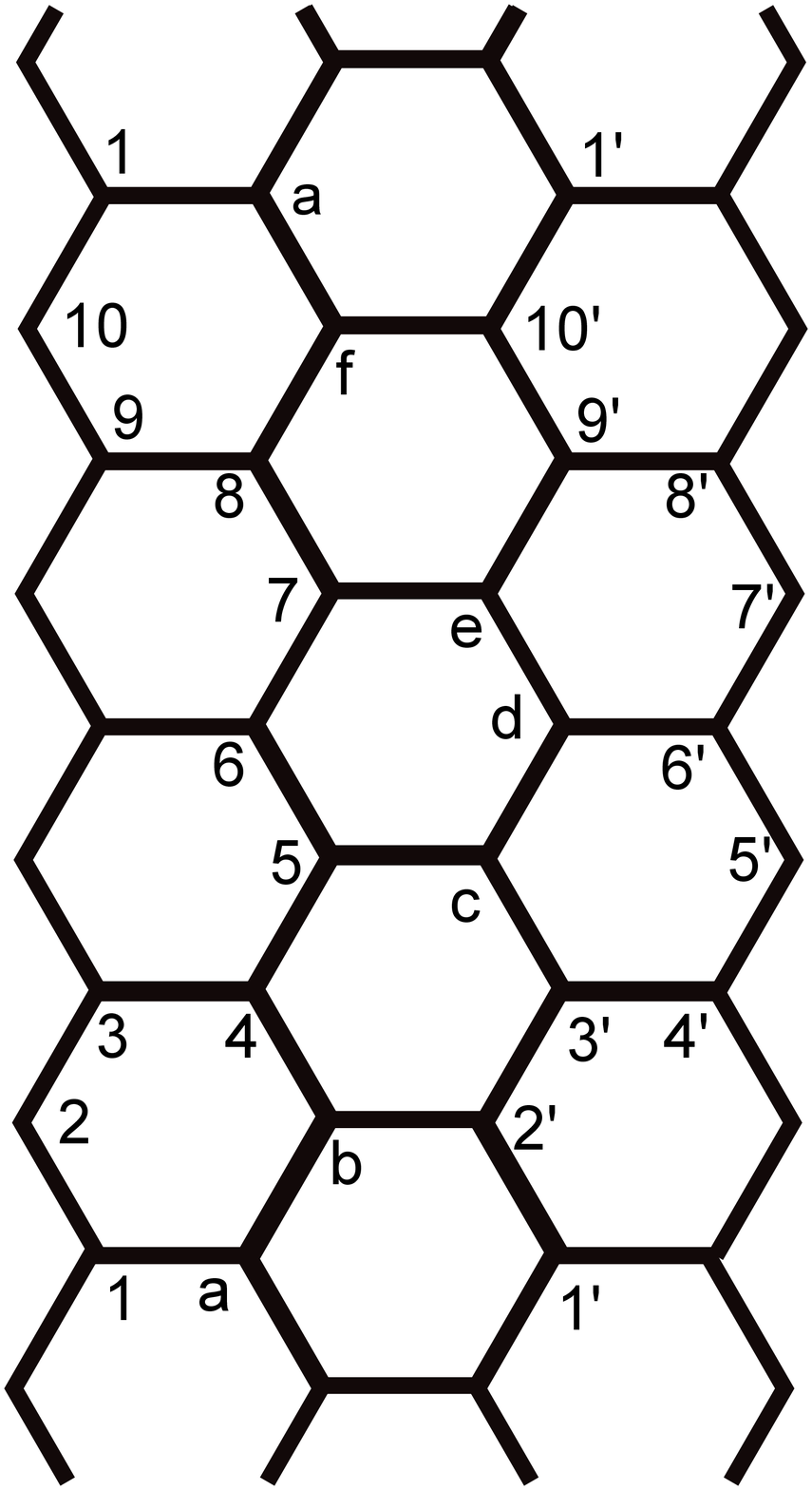}
\caption{(6,6,6) in $\mathds{R}^2$\label{3}}
\end{minipage}
\hfill
\begin{minipage}[t]{0.45\linewidth}
\centering
\includegraphics[width=\textwidth]{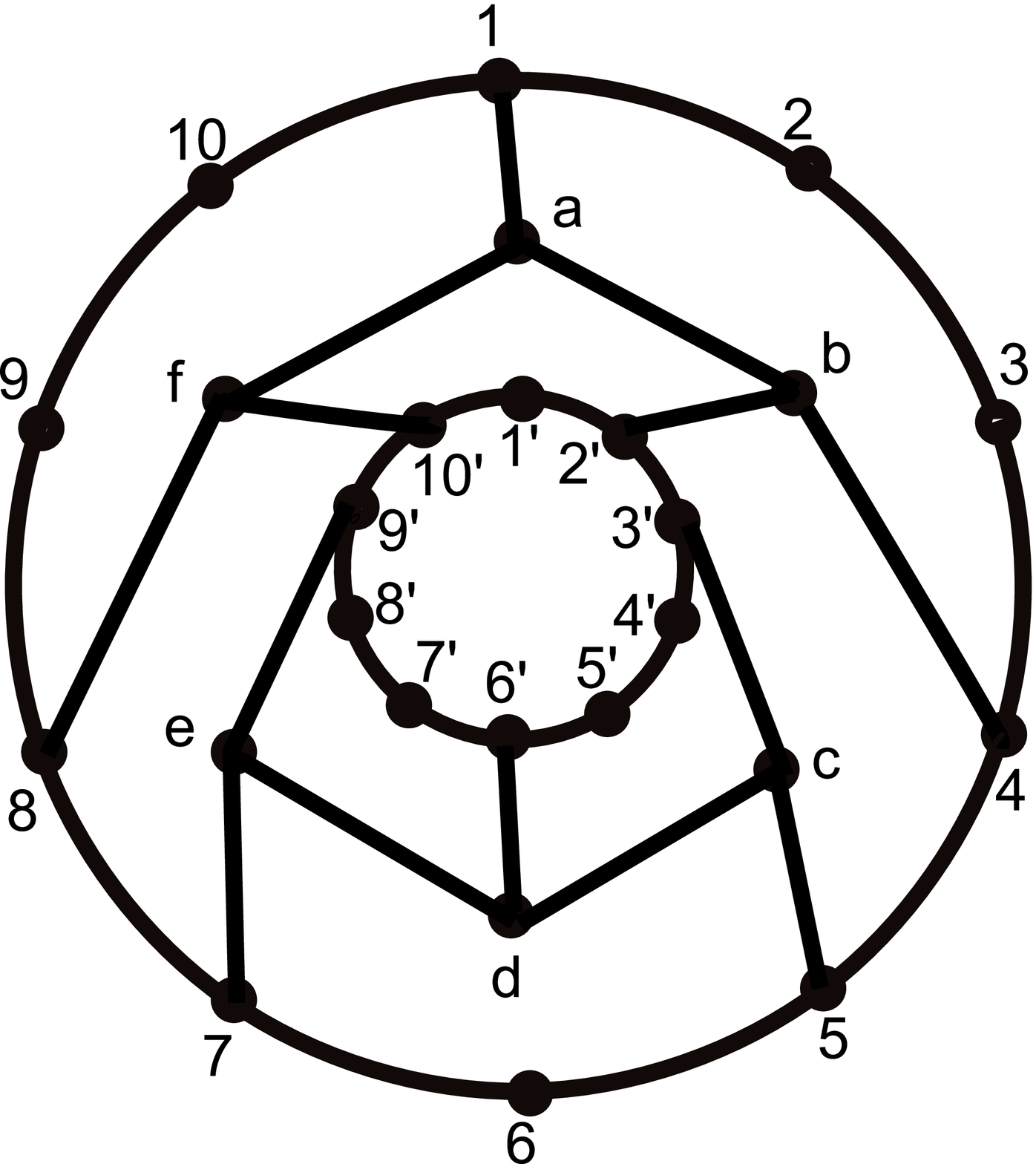}
\caption{(6,6,6) in
$\mathds{R}P^2\setminus\{\underline{o}\}$\label{2}}
\end{minipage}
\end{figure}

In addition, as the Archimedean tilings of the plane, we can
classify the tilings of the projective plane minus one point with
nonnegative curvature for which each vertex has the  same pattern.
\begin{theorem} The tilings of the projective plane minus one point
  with nonnegative curvature such that the pattern of each vertex is
  the same are induced by the Archimedean tilings of the plane and a gilde reflection symmetry.
\end{theorem}
We give two examples of tilings of the projective plane minus one
point which are induced by the Archimedean tilings and glide
reflection symmetries (see Figure \ref{5}, \ref{4}, \ref{7},
\ref{6}). It is easy to see that there are infinitely many tilings
of the projective plane minus one point with nonnegative curvature
because of the complexity of the tilings of the plane. Another way
to see the complexity is that we can apply the graph operation $P$
on the tiling of the projective plane minus one point with hexagonal
faces to obtain a new one.

\begin{figure}[h]
\begin{minipage}[t]{0.45\linewidth}
\centering
\includegraphics[width=0.8\textwidth]{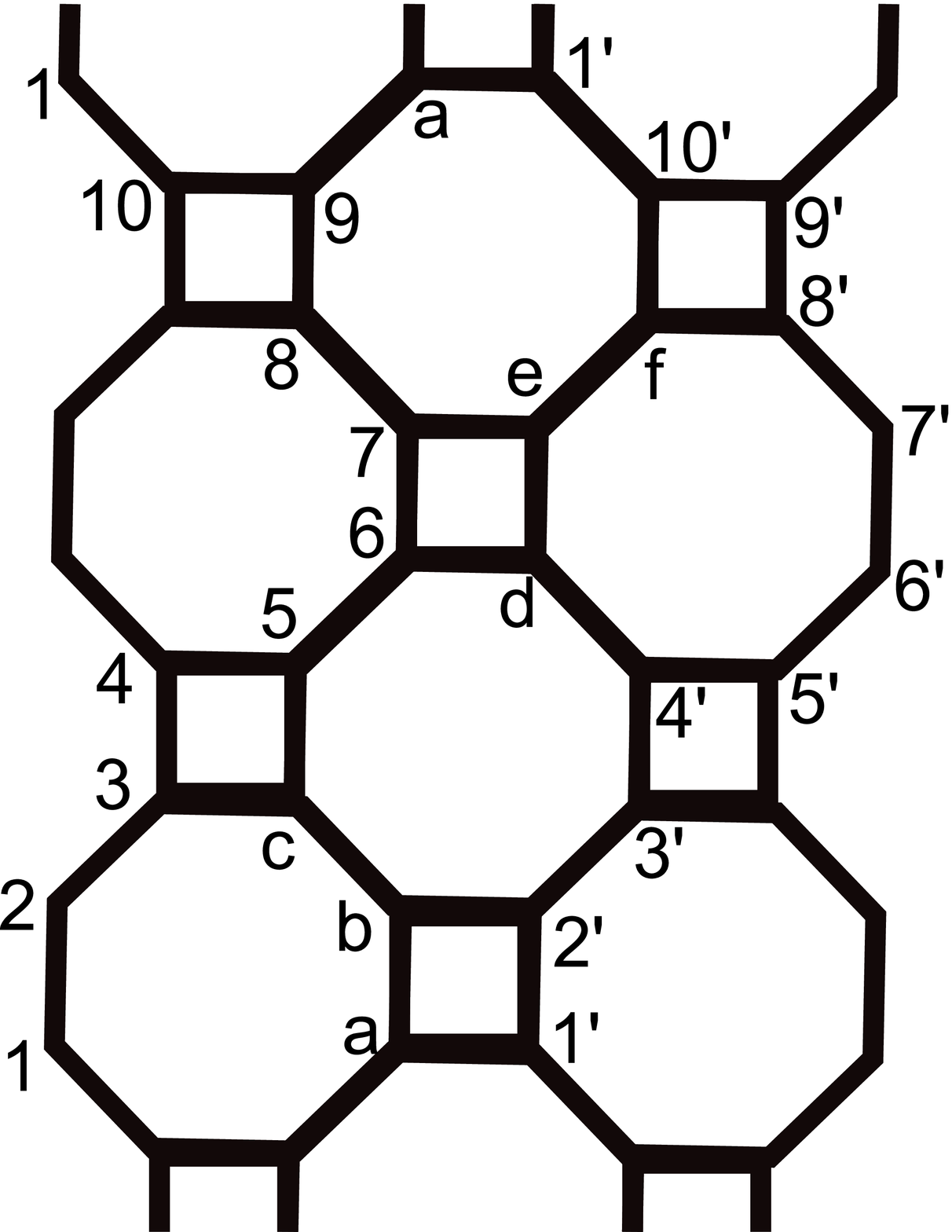}
\caption{(4,8,8) in $\mathds{R}^2$ \label{5}}
\end{minipage}
\hfill
\begin{minipage}[t]{0.45\linewidth}
\centering
\includegraphics[width=\textwidth]{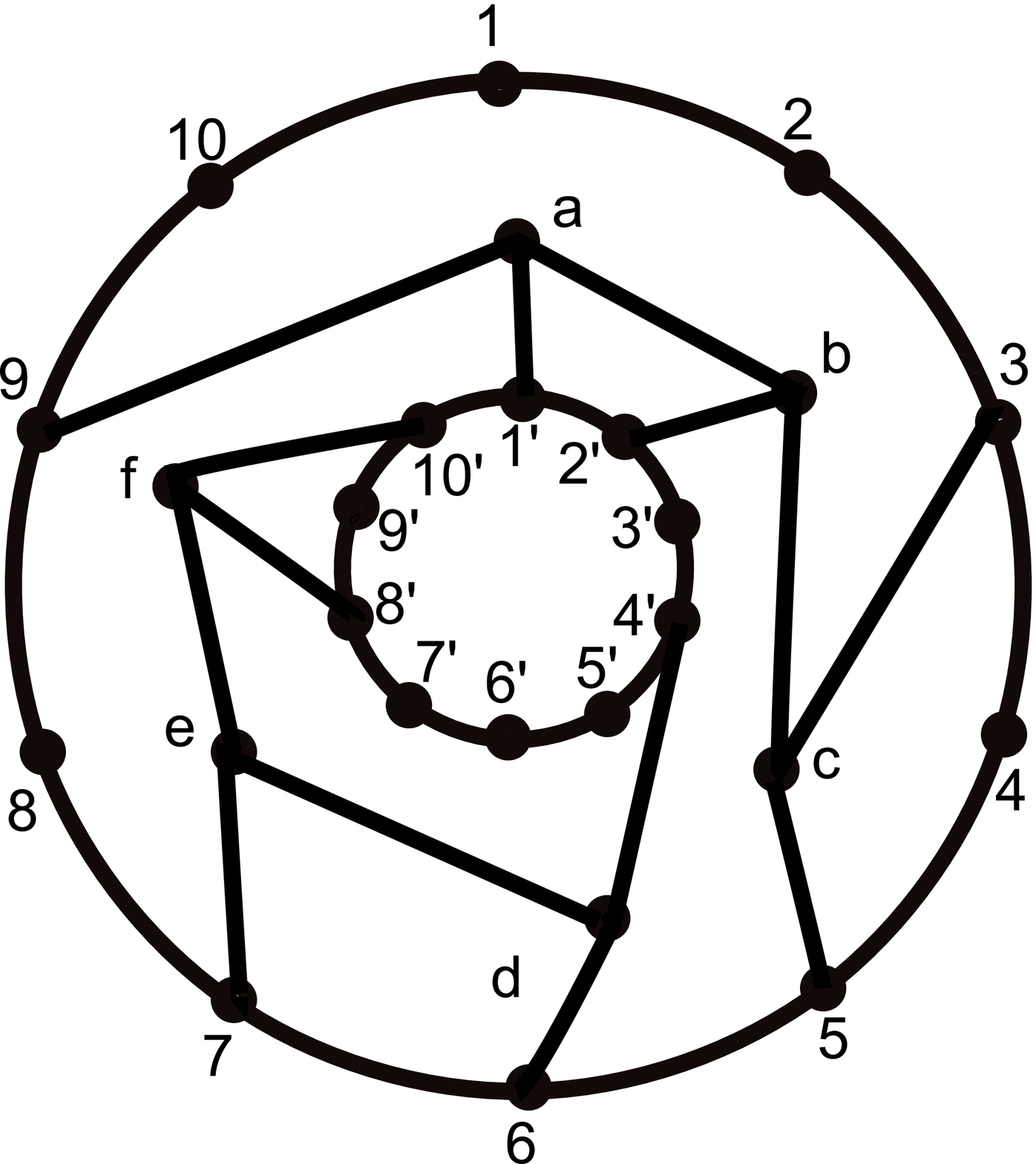}
\caption{(4,8,8) in
$\mathds{R}P^2\setminus\{\underline{o}\}$\label{4}}
\end{minipage}
\end{figure}

\begin{figure}[h]
\begin{minipage}[t]{0.45\linewidth}
\centering
\includegraphics[width=\textwidth]{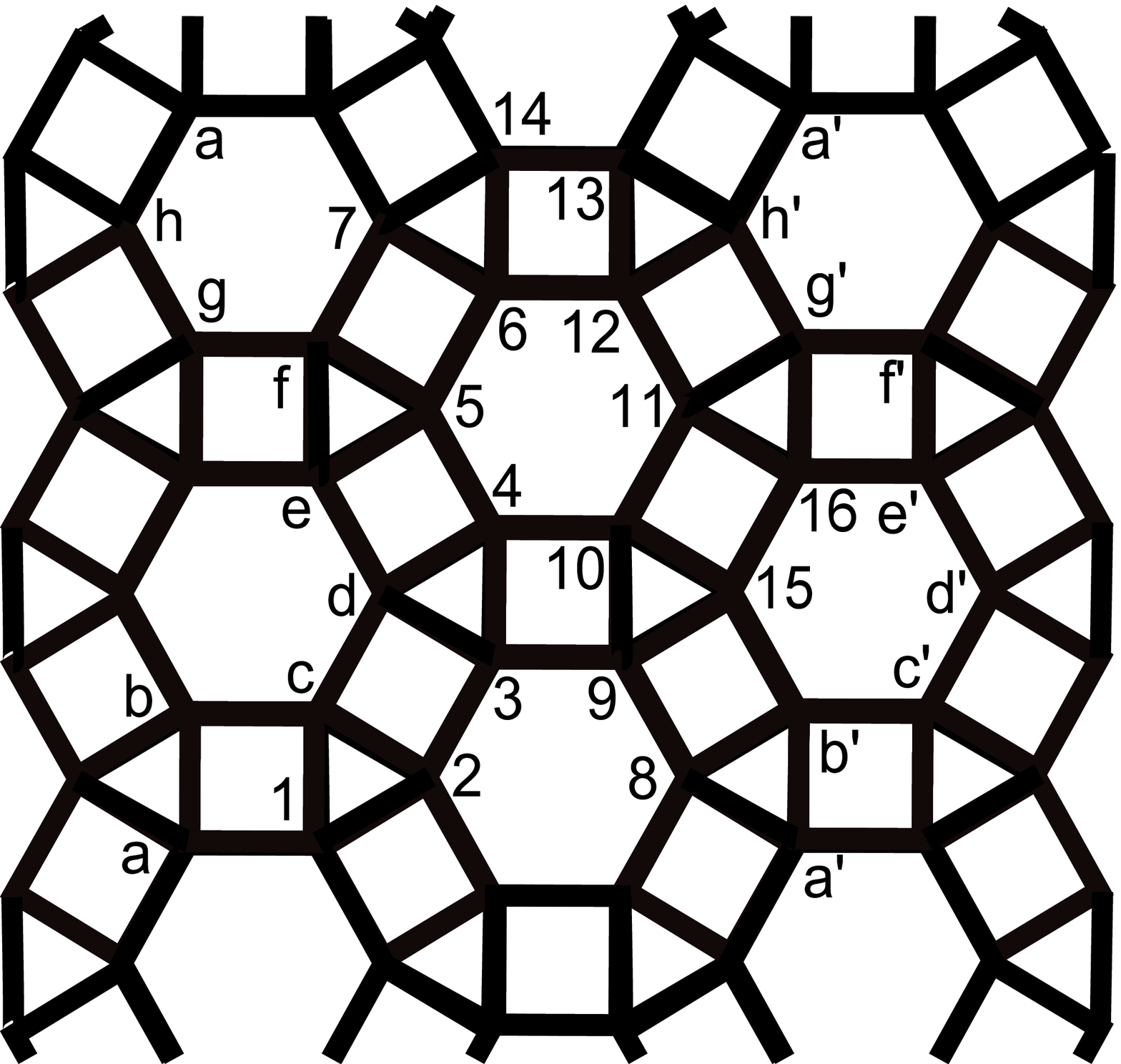}
\caption{(3,4,6,4) in $\mathds{R}^2$ \label{7}}
\end{minipage}
\hfill
\begin{minipage}[t]{0.45\linewidth}
\centering
\includegraphics[width=\textwidth]{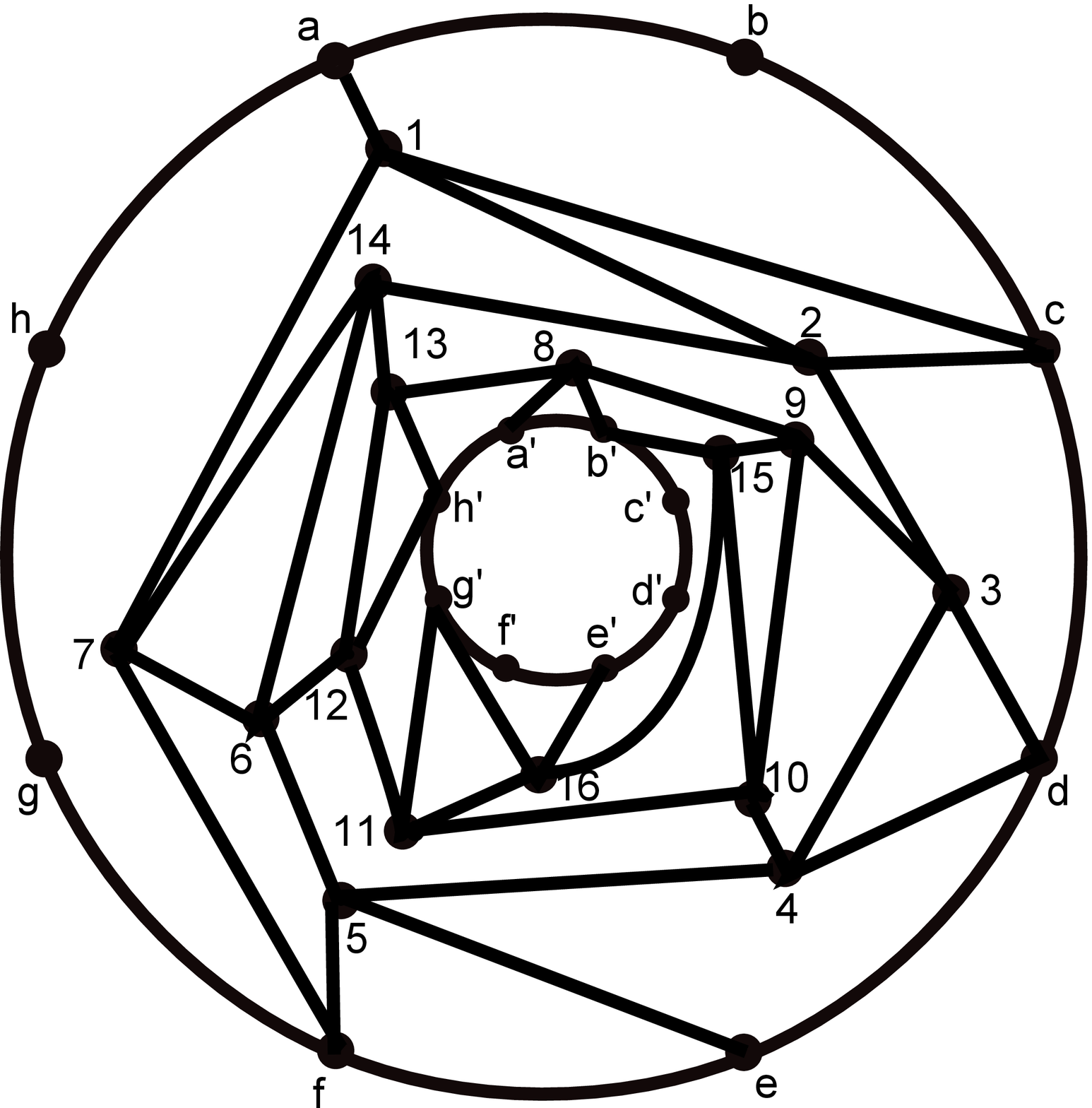}
\caption{(3,4,6,4) in
$\mathds{R}P^2\setminus\{\underline{o}\}$\label{6}}
\end{minipage}
\end{figure}

\section{Volume Doubling Property and Poincar\'e Inequality}
In this section, we shall prove the volume doubling property and the
Poincar\'e inequality for semiplanar graphs with nonnegative
curvature.

Let $G$ be a semiplanar graph  and $S(G)$ be the regular polygonal
surface of $G.$ For any $p\in G$ and $R>0,$ we denote by
$B_R(p)=\{x\in G: d^{G}(p,x)\leq R\}$ the closed geodesic ball in
the graph G, and by $B_R^{S(G)}(p)=\{x\in S(G): d(p,x)\leq R\}$ the
closed geodesic ball in the polygonal surface $S(G).$ The volume of
$B_R(p)$ is defined as $|B_R(p)|=\sum_{x\in B_R(p)}d_x,$ and the
volume of $B_R^{S(G)}(p)$ is defined as
$|B_R^{S(G)}(p)|=\mathcal{H}^2(B_R^{S(G)}(p)),$ where
$\mathcal{H}^2$ is the 2-dimensional Hausdorff measure. We denote by
$\sharp B_R(p)$ the number of vertices in the closed geodesic ball
$B_R(p).$ Note that for any semiplanar graph $G$ with nonnegative
curvature, $3\leq d_x\leq 6,$ for any $x\in G.$ Hence $|B_R(p)|$ and
$\sharp B_R(p)$ are equivalent up to a constant, i.e. $3 \sharp
B_R(p)\leq|B_R(p)|\leq 6 \sharp B_R(p),$ for any $p\in G$ and $R>0.$

\begin{theorem}\label{GRV} Let $G=(V,E,F)$ be a semiplanar graph with $\Sec (G)\geq0$. Then there exists a constant $C(D)$ depending on $D:=\max_{\sigma\in F}\deg(\sigma),$ such
that for any $p\in G$ and $0<r<R,$ we have
\begin{equation}\label{GRV1}\frac{|B_R(p)|}{|B_r(p)|}\leq C(D)\left(\frac{R}{r}\right)^2.\end{equation}
\end{theorem}
\begin{proof} We denote $B_R:=B_R(p)$ and $B_R^{S}:=B_R^{S(G)}(p)$ for short. By Lemma \ref{LEM1}, we have $B_{CR}^S\cap G\subset B_R\subset B_{R}^S\cap G.$ For any $\sigma\in F,$ $C_1\leq|\sigma|:=\mathcal{H}^2(\sigma)\leq C_2(D).$ Let $H_R:=\{\sigma\in F:\sigma \cap B_R\neq\emptyset\}$ denote the faces attached to $B_R.$ Then
\begin{equation}\label{GRVP1}|B_R|=\sum_{x\in B_R}d_x\leq D\cdot\sharp H_R,\end{equation} where $\sharp H_R$ is the number of faces in $H_R.$
For any $\sigma\in F,$ since the intrinsic diameter of $\sigma$ is
bounded, i.e. $\diam\ \sigma:=\sup\{d(x,y): x, y\in \sigma\}\leq D,$
we have for any face $\sigma\in H_R$
$$\sigma\subset B_{R+\diam \sigma}^S\subset B_{R+D}^S.$$ Hence it follows that
\begin{equation}\label{GRVP2}C_1\sharp H_R\leq\sum_{\sigma\in H_R}|\sigma|\leq|B_{R+D}^S|.\end{equation} By the volume comparison of $S(G)$ (\ref{VGA1}), (\ref{GRVP1}) and (\ref{GRVP2}), we obtain
\begin{equation}\label{GRVP3}|B_R|\leq C(D)|B_{R+D}^S|\leq C(D)(R+D)^2.\end{equation}
For $R\geq D,$ we have $|B_R|\leq 4C(D)R^2.$ For $1\leq R<D,$ we
have $|B_R|\leq C(D)R^2$ by \eqref{GRVP3}. Hence, for any $R\geq1,$
the quadratic volume growth property follows
\begin{equation}\label{VGG1}|B_R|\leq C(D)R^2.\end{equation}

For any $r>\frac{D}{C},$ where C is the constant in Lemma
\ref{LEM1}, let $r'=Cr-D.$ We denote by $W_r:=\{\sigma\in F:
\sigma\cap B_{r'}^S\neq\emptyset\}$ the faces attached to
$B_{r'}^S,$ and by $\overline{W_{r}}:=\bigcup_{\sigma\in W_r}\sigma$
the union of faces in $W_r$. For any vertex $x\in
\overline{W_{r}}\cap G,$ there exists a $\sigma\in W_r$ such that
$x\in\sigma$ and $d(p,x)\leq r'+\mathrm{diam} \sigma\leq r'+D=Cr.$
By Lemma \ref{LEM1}, we have $d^G(p,x)\leq C^{-1}d(p,x)\leq r,$
which implies that $\overline{W_{r}}\cap G\subset B_r.$ It is easy
to see that
\begin{equation}\label{GRVP4}|B_{r'}^S|\leq |\overline{W_{r}}|=\sum_{\sigma\in W_r}|\sigma|\leq C_2(D) \sharp W_r,\end{equation} where $\sharp W_r$ is the number of faces in $W_r.$
Moreover, by $3\leq \deg(\sigma)\leq D$ for any $\sigma\in F,$
\begin{equation}\label{GRVP5}3\sharp W_r\leq\sum_{\sigma\in W_r}\deg(\sigma)\leq\sum_{x\in \overline{W_r}\cap G} d_x\leq 6\sharp (\overline{W_r}\cap G),\end{equation} where $\sharp (\overline{W_r}\cap G)$ is the number of vertices in $\overline{W_r}\cap G.$

Hence by (\ref{GRVP4}) and (\ref{GRVP5}), we have
\begin{equation}\label{GRVP6}|B_{r'}^S|\leq C(D)\sharp (\overline{W_r}\cap G)\leq C(D)\sharp B_r\leq C(D)|B_r|.\end{equation}

By the relative volume comparison (\ref{RVCA1}), (\ref{GRVP3}) and
(\ref{GRVP6}), we obtain that for any $r>\frac{D}{C},$
$$\frac{|B_R|}{|B_r|}\leq C(D) \frac{|B_{R+D}^S|}{|B_{r'}^S|}\leq C(D)\left(\frac{R+D}{Cr-D}\right)^2.$$

Let $r_0(D):=\frac{2D}{C}.$ For $r_0(D)\leq r<R<\infty,$
$r-\frac{D}{C}\geq \frac{r}{2}$ and $R+D\leq 2R,$ so that we have
\begin{equation}\label{GRVP7}\frac{|B_R|}{|B_r|}\leq C(D)\left(\frac{R}{r}\right)^2.\end{equation}

For $0<r<R\leq r_0(D),$ by (\ref{VGG1}), we have
\begin{equation}\label{GRVP8}\frac{|B_R|}{|B_r|}\leq\frac{|B_{r_0(D)}|}{|B_0|}\leq \frac{1}{3}C(D)r_0^2(D)\leq C(D)\left(\frac{R}{r}\right)^2.\end{equation}

For $0<r<r_0(D)<R,$ by (\ref{VGG1}), we have
\begin{equation}\label{GRVP9}\frac{|B_R|}{|B_r|}\leq\frac{C(D)R^2}{|B_0|}\leq C(D)r^2\left(\frac{R}{r}\right)^2\leq C(D)r_0^2(D)\left(\frac{R}{r}\right)^2.\end{equation}

Hence it follows from (\ref{GRVP7}), (\ref{GRVP8}) and (\ref{GRVP9})
that for any $0<r<R,$

$$\frac{|B_R|}{|B_r|}\leq C(D)\left(\frac{R}{r}\right)^2.$$

\end{proof}

From the relative volume comparison, it is easy to obtain the volume
doubling property.

\begin{corollary}Let $G$ be a semiplanar graph with $\Sec (G)\geq0$. Then there exists a constant $C(D)$ depending on $D,$ such
that for any $p\in G$ and $R>0,$ we have
\begin{equation}\label{GVDP}|B_{2R}(p)|\leq C(D)|B_R(p)|.\end{equation}
\end{corollary}

In the rest of this section, we shall prove the Poincar\'e
inequality on a semiplanar graph with nonnegative curvature.

\begin{theorem}\label{GPIT} Let $G$ be a semiplanar graph with $\Sec (G)\geq 0$. Then there exist two constants $C(D)$ and $C$ such that for any $p\in G, R>0, f:B_{CR}(p)\rightarrow \mathds{R}$, we have
\begin{equation}\label{GPI}\sum_{x\in B_R(p)}(f(x)-f_{B_R})^2d_x\leq C(D)R^2\sum_{x,y\in B_{CR}(p);x\sim y}(f(x)-f(y))^2,\end{equation}
where $f_{B_R}=\frac{1}{|B_R(p)|}\sum_{x\in B_R(p)}f(x)d_x$, and
$x\sim y$ means $x$ and $y$ are neighbors.
\end{theorem}

For any function on a semiplanar graph $G$, $f:G\rightarrow
\mathds{R},$ we shall construct a local $W^{1,2}$ function, denoted
by $f_2,$ on $S(G)$ with controlled energy in two steps, and then by
the Poincar\'e inequality (\ref{PIA1}) on $S(G),$ we obtain the
Poincar\'e inequality on the graph $G.$ In step 1, by linear
interpolation, we extend $f$ to a piecewise linear function on
$G_1,$ $f_1:G_1\rightarrow \mathds{R}.$ In step 2, we extend $f_1$
to each face of $G.$ For any regular $n$-polygon $\triangle_n$ of
side length one, there is a bi-Lipschitz map
$$L_n:\triangle_n\rightarrow B_{r_n},$$ where $B_{r_n}$ is the
closed disk whose boundary is the circumscribed circle of
$\triangle_n$ of radius $r_n=\frac{1}{2\sin\frac{\alpha_n}{2}}$ (for
$\alpha_n=\frac{2\pi}{n}$). Without loss of generality, we may
assume that the origin $\underline{o}=(0,0)$ of $\mathds{R}^2$ is
the barycenter of $\triangle_n,$ the point $(x,y)=(r_n,0)\in
\mathds{R}^2$ is a vertex of $\triangle_n,$ and
$B_{r_n}=B_{r_n}(\underline{o}).$  Then in  polar coordinates, $L_n$
reads
$$L_n:\triangle_n \ni (r,\theta)\mapsto (\rho, \eta)\in B_{r_n}(\underline{o}),$$ where for $\theta\in[j\alpha_n,(j+1)\alpha_n],\ j=0,1,\cdots,n-1,$
$$\left\{
\begin{array}{ll}
\rho=\frac{r\cos\big(\theta-(2j+1)\frac{\alpha_n}{2}\big)}{\cos\frac{\alpha_n}{2}}&\\
\eta=\theta&
 \end{array}\right..$$ It maps the boundary of $\triangle_n$ to the boundary of $B_{r_n}(\underline{o})$. Direct calculation shows that $L_n$ is a
 bi-Lipschitz map, i.e. for any $x,y \in \triangle_n$ we have $C_1|x-y|\leq|L_nx-L_ny|\leq C_2|x-y|,$ where $C_1$ and $C_2$ do not depend on $n.$
Then for any $\sigma\in F,$ we denote $n:=\deg(\sigma).$ Let
$g:B_{r_n}(\underline{o})\rightarrow \mathds{R}$ satisfy the
following boundary value problem
$$\left\{
\begin{array}{ll}
\Delta g=0,&in\ \mathring{B}_{r_n}(\underline{o})\\
g|_{\partial B_{r_n}(\underline{o})}=f_1\circ L_n^{-1}&
\end{array}\right.,$$ where $\mathring{B}_{r_n}(\underline{o})$ is the open ball.
Then we define $f_2:S(G)\rightarrow \mathds{R}$ as
\begin{equation}\label{DFF2}f_2|_{\sigma}=g\circ L_n.\end{equation}
It can be shown that $f_2$ is local $W^{1,2}$ function on $S(G)$,
since the singular points of $S(G)$ are isolated (see \cite{KMS}).

We need to control the energy of $f_2$ by its boundary values. The
following lemma is standard. We denote by $B_1$ the closed unit disk
in $\mathds{R}^2.$

\begin{lemma}\label{B1ECL} For any Lipschitz function $h:\partial B_1\rightarrow \mathds{R},$ let $g:B_1\rightarrow\mathds{R}$ satisfy the following boundary value problem
$$\left\{
\begin{array}{ll}
\Delta g=0,&in\ \mathring{B}_1\\
g|_{\partial B_1}=h&
\end{array}\right..$$ Then we have
$$\int_{B_1}|\nabla g|^2\leq\int_{\partial B_1}h_{\theta}^2,$$
$$\int_{\partial B_1}h^2\leq C\left(\int_{B_1}g^2+\int_{\partial B_1}h_{\theta}^2\right),$$ where $h_{\theta}=\frac{\partial h}{\partial \theta}.$
\end{lemma}
\begin{proof}
Let $\frac{1}{\sqrt{2\pi}}, \frac{\sin n\theta}{\sqrt{\pi}},
\frac{\cos n\theta}{\sqrt{\pi}}$ (for $n=1,2,\cdots$) be the
orthonormal basis of $L^2(\partial B_1).$ Then $h:\partial
B_1\rightarrow \mathds{R}$ can be represented in $L^2(\partial B_1)$
by
$$h(\theta)=a_0\frac{1}{\sqrt{2\pi}}+\sum_{n=1}^{\infty}\left( a_n\frac{\cos n\theta}{\sqrt{\pi}}+b_n\frac{\sin n\theta}{\sqrt{\pi}}\right).$$
So the harmonic function $g$ with boundary value $h$ is
$$g(r,\theta)=a_0\frac{1}{\sqrt{2\pi}}+\sum_{n=1}^{\infty}\left( a_nr^n\frac{\cos n\theta}{\sqrt{\pi}}+b_nr^n\frac{\sin n\theta}{\sqrt{\pi}}\right).$$

Since $\Delta g=0,$ we have $\Delta g^2=2|\nabla g|^2,$ then
$$\int_{B_1}|\nabla g|^2=\frac{1}{2}\int_{B_1}\Delta g^2=\frac{1}{2}\int_{\partial B_1}\frac{\partial g^2}{\partial r},$$ which follows from integration by parts.
So that $$\int_{B_1}|\nabla g|^2=\int_{\partial B_1}g
g_r=\sum_{n=1}^{\infty}n(a_n^2+b_n^2).$$ In addition,
$$\int_{\partial
B_1}h_{\theta}^2=\sum_{n=1}^{\infty}n^2(a_n^2+b_n^2). $$ Hence,
\begin{equation}\label{B1EC}\int_{B_1}|\nabla g|^2\leq\int_{\partial
B_1}h_{\theta}^2.\end{equation}

The second part of the theorem follows from an integration by parts
and the H\"older inequality.
\begin{eqnarray*}
\int_{\partial B_1}h^2&=&\int_{\partial B_1}(h^2 x)\cdot x=\int_{B_1}\nabla\cdot(g^2x)\\
&=&2\int_{B_1}g^2+2\int_{B_1}g\nabla g\cdot x\\
&\leq&2\int_{B_1}g^2+2\left(\int_{B_1}g^2\right)^{\frac{1}{2}}\left(\int_{B_1}|\nabla g|^2\right)^{\frac{1}{2}}\ \ \  (by\ |x|\leq1)\\
&\leq&3\int_{B_1}g^2+\int_{B_1}|\nabla g|^2\\
&\leq&3\int_{B_1}g^2+\int_{\partial B_1}h_{\theta}^2.\ \ \ \ \ \ \ \
\ \  \  \ \ \ \ \  \ \ \ \ \ \ \ \ (by\ (\ref{B1EC}))
\end{eqnarray*}
\end{proof}

Note that for the semiplanar graph $G$ with nonnegative curvature
and any face $\sigma=\triangle_n$ of $G$, we have $3\leq n\leq D,$
$\frac{1}{\sqrt3} \leq r_n=\frac{1}{2\sin\frac{\pi}{n}}\leq
\frac{1}{2\sin\frac{\pi}{D}}=C(D).$ Then the scaled version of Lemma
\ref{B1ECL} reads
\begin{lemma} For $3\leq n\leq D$, and any Lipschitz function $h:\partial B_{r_n}\rightarrow \mathds{R},$ we denote by $g$ the harmonic function satisfying the Dirichlet boundary value problem
$$\left\{
\begin{array}{ll}
\Delta g=0,&in\ \mathring{B}_{r_n}\\
g|_{\partial B_{r_n}}=h&
\end{array}\right..$$ Then it holds that
$$\int_{B_{r_n}}|\nabla g|^2\leq C(D)\int_{\partial B_{r_n}}h_T^2,$$
$$\int_{\partial B_{r_n}}h^2\leq C(D)\left(\int_{B_{r_n}}g^2+\int_{\partial B_{r_n}}h_T^2\right),$$ where $T=\frac{1}{r_n}\partial_{\theta}$ is the unit tangent vector on the boundary $\partial B_{r_n}$ and $h_T$ is the directional derivative of $h$ in $T.$
\end{lemma}

The following lemma follows from the bi-Lipschitz property of the
map $L_n:\triangle_n\rightarrow B_{r_n}.$
\begin{lemma}\label{EX1} Let $\sigma$ be a face of degree $n,$ i.e. $\sigma=\triangle_n,$ in a semiplanar graph $G.$ Let $f_2|_{\sigma}$ be constructed as (\ref{DFF2}), then we have
\begin{equation}\label{EC1}\int_{\triangle_n}|\nabla f_2|^2\leq C(D)\int_{\partial \triangle_n} (f_1)_{T_n}^2,\end{equation}
\begin{equation}\label{EC2}\int_{\partial \triangle_n}f_1^2\leq C(D)\left(\int_{\triangle_n}f_2^2+\int_{\partial \triangle_n}(f_1)_{T_n}^2\right),\end{equation} where $T_n$ is the unit tangent vector on the boundary $\partial \triangle_n$ and $(f_1)_{T_n}$ is the directional derivative of $f_1$ in $T_n.$
\end{lemma}

Let $e\subset\triangle_n$ be an edge with two incident vertices, $u$
and $v$. By linear interpolation, we have
$$\int_ef_1^2=\int_0^1(tf(u)+(1-t)f(v))^2dt=\frac{1}{3}(f(u)^2+f(u)f(v)+f(v)^2),$$ hence
\begin{equation}\label{ECTG1}\frac{1}{6}(f(u)^2+f(v)^2)\leq\int_ef_1^2\leq\frac{1}{2}(f(u)^2+f(v)^2).\end{equation} In addition,
\begin{equation}\label{ECTG2}\int_e(f_1)_{T_n}^2=(f(u)-f(v))^2.\end{equation} Now we can prove the Poincar\'e inequality.

\begin{proof}[Proof of Theorem \ref{GPIT}] Let $B_R^{G_1}:=B_R^{G_1}(p)$ denote the closed geodesic ball in $G_1.$ We set the constant
$c_R=(f_2)_{B_{R+1+D}^S}=\frac{1}{|B_{R+1+D}^S|}\int
_{B_{R+1+D}^S}f_2.$ By the combinatorial relation of vertices and
edges and (\ref{ECTG1}), we have
\begin{equation}\label{PIP0}\sum_{x\in
B_R}(f(x)-c_R)^2d_x\leq \sum_{\substack{e=uv\in E\\
e\cap
B_R\neq\emptyset}}[(f(u)-c_R)^2+(f(v)-c_R)^2]\leq6\int_{B_{R+1}^{G_1}}(f_1-c_R)^2.\end{equation}
Let $W_{R+1}=\{\sigma\in F: \sigma\cap B_{R+1}^{G_1}\neq\emptyset\}$
and $\overline{W_{R+1}}=\bigcup_{\sigma\in W_{R+1}}\sigma.$ Since
$B_{R+1}^{G_1}\subset\bigcup_{\sigma\in W_{R+1}}\partial\sigma,$ we
have
\begin{eqnarray}\label{PIP1}
\int_{B_{R+1}^{G_1}}(f_1-c_R)^2&\leq&\sum_{\sigma\in W_{R+1}}\int_{\partial\sigma}(f_1-c_R)^2\\
&\leq& C(D)\sum_{\sigma\in
W_{R+1}}\left(\int_{\sigma}(f_2-c_R)^2+\int_{\partial
\sigma}(f_1)_T^2\right),\nonumber
\end{eqnarray}
where the last inequality follows from (\ref{EC2}). For any
$y\in\overline{W_{R+1}},$ since $\mathrm{diam}\ \sigma\leq
\deg(\sigma)\leq D$ for any $\sigma\in F,$ we have $d(p,y)\leq
R+1+D.$ It implies that $\overline{W_{R+1}}\subset B_{R+1+D}^S.$
Hence by (\ref{PIP1})
\begin{eqnarray}\label{PIP2}
\int_{B_{R+1}^{G_1}}(f_1-c_R)^2&\leq&C(D)\int_{B_{R+1+D}^S}(f_2-c_R)^2+C(D)\sum_{\sigma\in W_{R+1}}\int_{\partial\sigma}(f_1)_T^2\nonumber\\
&\leq& C(D)(R+1+D)^2\int_{B_{R+1+D}^S}|\nabla f_2|^2\\
&&+C(D)\sum_{\sigma\in W_{R+1}}\int_{\partial\sigma}(f_1)_T^2,
\nonumber\end{eqnarray} where we use the Poincar\'e inequality
(\ref{PIA1}).

Let $U_{R+1}:=\{\tau\in F:\tau\cap B_{R+1+D}^S\neq\emptyset\}.$
Since $B_{R+1}^{G_1}\subset B_{R+1}^S\subset B_{R+1+D}^S,$ we have
$W_{R+1}\subset U_{R+1}.$ By Lemma \ref{LEM1}, it follows that
\begin{equation}\label{PIP3}U_{R+1}\cap G_1\subset
B_{C^{-1}(R+1+2D)}^{G_1}.\end{equation} By (\ref{EC1}),
(\ref{ECTG2}), (\ref{PIP2}) and (\ref{PIP3}), we obtain that
\begin{eqnarray}\label{PIP4}
\int_{B_{R+1}^{G_1}}(f_1-c_R)^2&\leq& C(D)(R+1+D)^2\sum_{\tau\in U_{R+1}}\int_{\tau}|\nabla f_2|^2\nonumber\\
&&+C(D)\sum_{\sigma\in W_{R+1}}\int_{\partial\sigma}(f_1)_T^2\nonumber\\
&\leq&C(D)(R+1+D)^2\sum_{\tau\in U_{R+1}}\int_{\partial\tau}(f_1)_T^2\nonumber\\
&&+C(D)\sum_{\sigma\in W_{R+1}}\int_{\partial\sigma}(f_1)_T^2\nonumber\\
&\leq&C(D)(R+1+D)^2\sum_{\tau\in U_{R+1}}\int_{\partial\tau}(f_1)_T^2\nonumber\\
&\leq&C(D)(R+1+D)^2\displaystyle{\sum_{\substack{x,y\in
B_{C^{-1}(R+1+2D)}\\x\sim y}}(f(x)-f(y))^2}.
\end{eqnarray}

For $R\geq C^{-1}(1+2D)=:r_0(D),$ we have $C^{-1}(R+1+2D)\leq
(C^{-1}+1)R=C_1 R$ and $R+1+2D\leq 2R.$ Let
$f_{B_R}=\frac{1}{|B_R|}\sum_{x\in B_R}f(x)d_x,$ then by
(\ref{PIP0}) and (\ref{PIP4}) we obtain
\begin{equation}\label{PIP5}\sum_{x\in B_R}(f(x)-f_{B_R})^2d_x\leq\sum_{x\in B_R}(f(x)-c_R)^2d_x\leq C(D)R^2\displaystyle{\sum_{\substack{x,y\in B_{C_1R}\\x\sim y}}(f(x)-f(y))^2}.\end{equation}

For $1\leq R\leq r_0(D),$ let $G^R=(V^R, E^R)$ be the subgraph
induced by $B_R.$ For any $x\in G^R,$ we denote by $d_{x,G^R}$ the
degree of the vertex $x$ in $G^R.$ The volume of $G^R$ is defined as
$\vol G^R=\sum_{x\in G^R}d_{x,G^R}$ and the diameter of $G^R$ is
defined as $\mathrm{diam} G^R=\sup_{x,y\in G^R}d^{G^R}(x,y).$  Let
$\lambda_1(G^R)$ be the first nonzero eigenvalue of the Laplacian of
$G^R,$ then the Rayleigh principle implies that
$$\lambda_1(G^R)=\inf_{\substack{f:G^R\rightarrow \mathds{R}\\ f\neq\mathrm{const.}}}\frac{\sum_{x,y\in G^R;x\sim y}(f(x)-f(y))^2}{\sum_{x\in G^R}(f(x)-f_{G^R})^2d_{x,G^R}},$$ where
$f_{G^R}=\frac{1}{\vol G^R}\sum_{x\in G^R}f(x)d_{x,G^R}.$ We recall
a lower bound estimate for $\lambda_1(G^R)$ by the diameter and
volume of $G^R$ (see \cite{Chu}),
$$\lambda_1(G^R)\geq\frac{1}{\diam G^R\cdot \vol G^R}.$$
Since $3\leq d_x\leq 6,$ we have $\frac{1}{6}d_x\leq d_{x,G^R}\leq
d_x.$ It is easy to see that $\diam G^R\leq 2R$ and $\vol
G^R\leq|B_R|\leq C(D)R^2$ by (\ref{VGG1}). So that we have
$$\lambda_1(G^R)\geq\frac{1}{2R\cdot C(D)R^2}\geq\frac{1}{2r_0(D)\cdot C(D)r_0^2(D)}\geq C(D),$$ which implies that
$${\sum_{x\in G^R}(f(x)-f_{G^R})^2d_{x,G^R}}\leq C(D)\sum_{x,y\in G^R;x\sim y}(f(x)-f(y))^2,$$ for any $f: G^R\rightarrow \mathds{R}.$

Hence we obtain that
\begin{eqnarray}\label{PIP6}
\sum_{x\in B_R}(f(x)-f_{B_R})^2d_x&\leq&6\sum_{x\in G^R}(f(x)-f_{G^R})^2d_{x,G^R}\nonumber\\
&\leq&C(D)\sum_{x,y\in G^R;x\sim y}(f(x)-f(y))^2\nonumber\\
&\leq&C(D)R^2\sum_{x,y\in B_R;x\sim y}(f(x)-f(y))^2.
\end{eqnarray}

For $0<R<1,$ the Poincar\'e inequality (\ref{GPI}) is trivial. The
theorem is proved by (\ref{PIP5}) and (\ref{PIP6}).
\end{proof}

\section{Analysis on Semiplanar Graphs with Nonnegative Curvature}
In this section, we shall study the analytic consequences of the
volume doubling property and the Poincar\'e inequality.

In Riemannian manifolds, it is well known that the volume doubling
property and the Poincar\'e inequality are sufficient for the
Nash-Moser iteration which implies the Harnack inequality for
positive harmonic functions (see \cite{G1,S1}).

Let $G$ be a graph. For a function $f:G\rightarrow \mathds{R}$, the
Laplace operator $L$ is defined as
$$Lf(x)=\frac{1}{d_x}\sum_{y\sim x}(f(y)-f(x)).$$ The gradient of
$f$ is defined as $$|\nabla f|^2(x)=\sum_{y\sim x}(f(y)-f(x))^2.$$
Given a subset $\Omega\subset G,$ a function $f$ is called harmonic
(subharmonic, superharmonic) on $\Omega$ if $Lf(x)=0 (\geq0,\leq0)$
for any $x\in\Omega.$ We denote by $H^d(G)=\{f: Lf\equiv0,
|f(x)|\leq C(d^G(p,x)+1)^d\}$ the space of polynomial growth
harmonic functions of growth degree less than or equal to $d$ on
$G$.

It was proved by Delmotte \cite{D2} and Holopainen-Soardi \cite{HS}
independently that the Harnack inequality for positive harmonic
functions holds on graphs satisfying the volume doubling property
and the Poincar\'e inequality. Applying their results to our case,
we obtain the following theorem.

\begin{theorem}[\cite{D2,HS}] Let $G$ be a semiplanar graph
  with $\Sec (G)\geq0$. Then there exist constants
  $C_1 >1$, $C_2(D)<\infty$ such that for any $R>0,p\in G$ and any positive harmonic function $u$ on $B_{C_1R}(p),$ we have
\begin{equation}\label{HARNK}\max_{B_R(p)}u\leq C_2(D)\min_{B_R(p)}u.\end{equation}
\end{theorem}

\begin{remark} \emph{In \cite{D3}, Delmotte obtained the parabolic
    Harnack inequality and the Gaussian estimate for the heat kernel
    which is stronger than the elliptic one of the preceding theorem.}
\end{remark}

In the Nash-Moser iteration, the mean value inequality for
nonnegative subharmonic functions is obtained (see \cite{CoG}).
Since the square of a harmonic function is subharmonic, we obtain
\begin{lemma}
Let $G$ be a semiplanar graph with $\Sec (G)\geq0$. Then there exist
two constants $C_1$ and $C_2(D)$ such that for any $R>0,p\in G,$ any
harmonic function $u$ on $B_{C_1R}(p),$ we have
\begin{equation}\label{MVI}u^2(p)\leq\frac{C_2(D)}{|B_{C_1R}(p)|}\sum_{x\in B_{C_1R}(p)}u^2(x)d_x.\end{equation}
\end{lemma}

The Liouville theorem for positive harmonic functions follows from
the Harnack inequality (see \cite{S2}).
\begin{theorem} Let $G$ be a semiplanar graph with $\Sec (G)\geq0.$ Then any positive harmonic function on $G$ must be constant. \end{theorem}
\begin{proof} Since $G$ be a semiplanar graph with $\Sec (G)\geq0,$ then $D_G<\infty.$ Let $u$ be a positive harmonic function on $G.$ By the Harnack inequality (\ref{HARNK}), we obtain \begin{equation}\label{POFL}\max_{B_R}(u-\inf_{G}u)\leq C_2(D_G)\min_{B_R}(u-\inf_Gu),\end{equation} for any $R>0.$ The right hand side of (\ref{POFL}) tends to $0$ if $R\rightarrow\infty$.
Hence, $$u\equiv\inf_Gu=\mathrm{const}.$$

\end{proof}

A manifold or a graph is called parabolic if it does not admit any
nontrivial positive superharmonic function. The parabolicity of a
manifold has been extensively studied in the literature (see
\cite{G2,HK,RSV}). In the graph setting, it is equivalent to the
fact that the simple random walk on the graph is recurrent, see e.g.
\cite{Woe1}. We already prove the quadratic volume growth,
\eqref{VGG1} in Theorem \ref{GRV}, of the semiplanar graph with
nonnegative curvature. Lemma 3.12 in \cite{Woe1} yields the
parabolicity of such graphs.

\begin{theorem} Any semiplanar graph $G$ with $\Sec (G)\geq 0$ is parabolic.
\end{theorem}

In the last part of the section, we investigate the polynomial
growth harmonic function theorem on  graphs. For Riemannian
manifolds, the polynomial growth harmonic function theorem was
proved by Colding and Minicozzi in \cite{CM2}. By assuming the
volume doubling property (\ref{GVDP}) and the Poincar\'e inequality
(\ref{GPI}) on the graph, Delmotte \cite{D1} proved the polynomial
growth harmonic function theorem with the dimension estimate in our
case
$$\dim H^d(G)\leq C(D)d^{v(D)},$$ where $C(D)$ and $v(D)$ depend on the maximal facial degree $D$ of the semiplanar graph $G$ with nonnegative curvature.
We improve Delmotte's dimension estimate of $H^d(G)$ by using the
relative volume comparison (\ref{GRV1}) instead of the volume
doubling property (\ref{GVDP}).
\begin{theorem}\label{PNG1T} Let $G$ be a semiplanar graph with $\Sec (G)\geq0.$ Then
\begin{equation}\label{PNG1}\dim H^d(G)\leq C(D)d^2,\end{equation} for any $d\geq1.$
\end{theorem}

We will use the argument by the mean value inequality (see
\cite{CM4,L1}). From now on, we fix some vertex $p\in G,$ and denote
$B_R=B_R(p)$ for short. We need the following lemmas.
\begin{lemma}\label{Fdl} For any finite dimensional subspace $K\subset H^d(G),$ there exists a constant $R_0(K)$ depending on $K$ such that for any $R\geq R_0(K),$ $u,v\in K,$ $$A_R(u,v):=\sum_{x\in B_R}u(x)v(x)d_x$$ is an inner product on $K.$
\end{lemma}
\begin{proof} The lemma is easily proved by a contradiction argument (see \cite{Hu2}).
\end{proof}
\begin{lemma}\label{PNGL1} Let $G$ be a semiplanar graph with $\Sec (G)\geq0,$ $K$ be a $k$-dimensional subspace of $H^d(G).$ Given $\beta>1,\delta>0,$ for any $R_1\geq R_0(K)$ there exists $R>R_1$ such that if $\{u_i\}_{i=1}^k$ is an orthonormal basis of $K$ with respect to the inner product $A_{\beta R},$ then $$\sum_{i=1}^k A_R(u_i,u_i)\geq k\beta^{-(2d+2+\delta)}.$$

\end{lemma}
\begin{proof} The proof is same as \cite{L2,D1,Hu2}.
\end{proof}

\begin{lemma}\label{PNGL2} Let $G$ be a semiplanar graph with $\Sec (G)\geq0,$ $K$ be a $k$-dimensional subspace of $H^d(G).$ Then there exists a constant
$C(D)$ such that for any basis of $K,$ $\{u_i\}_{i=1}^k,$ $R>0,
0<\epsilon<\frac{1}{2},$ we have
$$\sum_{i=1}^kA_R(u_i,u_i)\leq C(D)\epsilon^{-2}\sup_{u\in <A,U>}\sum_{y\in B_{(1+\epsilon)R}}u^2(y)d_y,$$ where $<A,U>:=\{w=\sum_{i=1}^ka_iu_i:\sum_{i=1}^ka_i^2=1\}.$
\end{lemma}
\begin{proof} For any $x\in B_R,$ we set $K_x=\{u\in K: u(x)=0\}.$ It
  is easy to see that $\dim K/K_x\leq1.$ Hence there exists an orthonormal linear transformation $\phi:K\rightarrow K,$ which maps $\{u_i\}_{i=1}^k$ to $\{v_i\}_{i=1}^k$ such that $v_i\in K_x,$ for $i\geq2.$ The mean value inequality (\ref{MVI}) implies that
\begin{eqnarray*}
\sum_{i=1}^ku_i^2(x)&=&\sum_{i=1}^kv_i^2(x)=v_1^2(x)\\
&\leq&C(D)|B_{(1+\epsilon)R-r(x)}(x)|^{-1}\sum_{y\in B_{(1+\epsilon)R-r(x)}(x)}v_1^2(y)d_y\\
&\leq&C(D)|B_{(1+\epsilon)R-r(x)}(x)|^{-1}\sup_{u\in
<A,U>}\sum_{y\in B_{(1+\epsilon)R}}u^2(y)d_y,
\end{eqnarray*} where $r(x)=d^G(p,x).$

By the relative volume comparison (\ref{GRV1}), we have
\begin{eqnarray*}|B_{(1+\epsilon)R-r(x)}|&\geq& C(D)\left(\frac{(1+\epsilon)R-r(x)}{2R}\right)^2|B_{2R}(x)|\\
&\geq&C(D)\left(\frac{(1+\epsilon)R-r(x)}{2R}\right)^2|B_{R}(p)|\geq
C(D)\epsilon^2|B_R|.
\end{eqnarray*}
Hence, by $3\leq d_x\leq6$ for any $x\in G,$
$$\sum_{i=1}^k\sum_{x\in B_R}u_i^2(x)d_x\leq 6\sum_{i=1}^k\sum_{x\in B_R}u_i^2(x)\leq C(D)\epsilon^{-2}\sup_{u\in <A,U>}\sum_{y\in B_{(1+\epsilon)R}}u^2(y)d_y.$$

\end{proof}

\begin{proof}[Proof of Theorem \ref{PNG1T}] For any $k$-dimensional subspace $K\subset H^d(G),$ we set $\beta=1+\epsilon.$ By Lemma \ref{PNGL1}, there exists $R>R_0(K)$ such that for any orthonormal basis $\{u_i\}_{i=1}^k$ of $K$ with respect to $A_{(1+\epsilon) R},$ we have
$$\sum_{i=1}^k A_R(u_i,u_i)\geq k(1+\epsilon)^{-(2d+2+\delta)}.$$
Lemma \ref{PNGL2} implies that $$\sum_{i=1}^k A_R(u_i,u_i)\leq
C(D)\epsilon^{-2}.$$ Setting $\epsilon=\frac{1}{2d},$ and letting
$\delta\rightarrow 0,$ we obtain
$$k\leq C(D)\left(\frac{1}{2d}\right)^{-2}\left(1+\frac{1}{2d}\right)^{2d+2+\delta}\leq C(D)d^2.$$
\end{proof}
The dimension estimate in (\ref{PNG1}) is not satisfactory since in
Riemannian geometry the constant $C(D)$  depends only on the
dimension of the manifold rather than the maximal facial degree of
$G$. Note that Theorem \ref{LFC} shows that the semiplanar graph $G$
with $\Sec (G)\geq0$ and $D_G\geq43$ has a special structure, i.e.
the one-side cylinder structure of linear volume growth. In
Riemannian geometry, Sormani \cite{Sor} used Yau's gradient estimate
and the nice behavior of the Busemann function on a one-end
Riemannian manifold with nonnegative Ricci curvature of linear
volume growth to show that it does not admit any nontrivial
polynomial growth harmonic function. Inspired by the work \cite{Sor}
and the special structure of semiplanar graphs with nonnegative
curvature and large face degree, we shall prove the following
theorem.

\begin{theorem}\label{PGTF} Let G be a semiplanar graph with $\Sec (G)\geq 0$ and $D_G\geq 43$. Then for any $d>0$,
$$\dim H^d(G)=1.$$
\end{theorem}

To prove the theorem, we need a weak version of the gradient
estimate \cite{LX}. We recall the Cacciappoli inequality for
harmonic functions on the graph $G.$
\begin{theorem}[\cite{LX}]\label{LYT} Let $G$ be a graph and $d_m=\sup_{x\in G}d_x.$ For any harmonic function $u$ on $B_{6r},$ $r\geq1,$ we have
$$\sum_{x\in B_r}|\nabla u|^2(x)\leq\frac{C(d_m)}{r^2}\sum_{y\in B_{6r}}u^2(y)d_y,$$
Moreover for any $x\in B_r,$
\begin{equation}\label{LY}|\nabla u|^2(x)\leq\frac{C(d_m)}{r^2}\sum_{y\in B_{6r}}u^2(y)d_y.\end{equation}
\end{theorem}

\begin{corollary}\label{WGEC} Let $G$ be a semiplanar graph with $\Sec (G)\geq 0$ and $D_G\geq 43$. For any harmonic function $u$ on $G$, we have
\begin{equation}\label{WGE}|\nabla u|(x)\leq \frac{C(D)}{\sqrt r}\mathrm{osc}_{B_{6r}(x)}u,\end{equation} where $\mathrm{osc}_{B_{6r}(x)}u=\max_{B_{6r}(x)}u-\min_{B_{6r}(x)}u.$
\end{corollary}
\begin{proof} By Theorem \ref{LFC}, the regular polygonal surface $S(G)$ has linear volume growth. As same as in the proof of (\ref{VGG1}) in Theorem \ref{GRV}, we obtain that  for any $x\in G$ and $r\geq 1,$
\begin{equation}\label{VLG}|B_r(x)|\leq C(D)r.\end{equation} By (\ref{LY}) in Theorem \ref{LYT} and $d_m\leq6$, we have
\begin{equation}\label{PCGE1}|\nabla u|^2(x)\leq \frac{C}{r^2}\sum_{y\in B_{6r}(x)}u^2(y)d_y\leq \frac{C}{r^2}|B_{6r}(x)|\max_{B_{6r}(x)}|u|^2.\end{equation}
We replace $u$ by $u-\min_{B_{6r}(x)}u$ in (\ref{PCGE1}), noting
that (\ref{VLG}), to obtain that
$$|\nabla u|(x)\leq \frac{C(D)}{\sqrt r}\mathrm{osc}_{B_{6r}(x)}u.$$
\end{proof}

\begin{remark} \emph{We call} \emph{(\ref{LY})} \emph{the weak version
    of the gradient estimate since its scaling is not as usual, but it suffices for our application.}
\end{remark}

\begin{proof}[Proof of Theorem \ref{PGTF}] Let $G$ be a semiplanar
  graph with $\Sec(G)\geq 0$ and $D_G\geq 43.$ Let $\sigma$ be the
  largest face with $\deg(\sigma)=D_G=D\geq 43.$ By Theorem \ref{LFC},
  either $G$ has no hexagons, which looks like $\sigma, L_1,L_2,\cdots,L_m,\cdots$ where each
  $L_m$ has the same type of faces (triangle or square), i.e. $G=\sigma\cup\bigcup_{m=1}^{\infty}L_m,$ or $G$ has hexagons, $G=P^{-1}(\sigma\cup\bigcup_{m=1}^{\infty}L_m),$ where $P^{-1}$ is the graph operation defined in section 2.
Denote by $A=\sigma\cap G$ the set of vertices incident to $\sigma,$
by $d^G(x,A)=\min_{y\in A}d^G(x,y)$ the distance function of $A$ in
$G.$ Let $B_r(A)=\{x\in G:d^G(x,A)\leq r\}$ and $\partial
B_r(A)=\{x\in G:d^G(x,A)=r\}.$

We claim that by the construction of $G,$ for any $x,y\in \partial
B_r(A),$ there is a path joining $x$ and $y$ in $B_r(A)$ with length
less than or equal to $3D.$ For the case that $G$ has no hexagons,
we know that $\sharp
\partial B_r(A)=D$ and $\partial B_r(A)$ is a closed path (a cycle).
This is trivial. For the case that $G$ has hexagons, by induction on
$r\in\mathrm{N},$ we have $\sharp
\partial B_r(A)\leq D$ and the distance to $A,$ i.e. $d(x,A)$ is invariant under the operation $P^{-1}.$
Since for each two vertices in $\partial B_r(A)$ sharing one
hexagon, there is a path on the boundary of the hexagon with length
$\leq 3$ which lies in $B_r(A).$ Note that $\sharp
\partial B_r(A)\leq D,$ the claim follows.

In addition, for any $q\in A,$ we have
\begin{equation}\label{FTP1}\partial B_r(A)\subset
B_{r+D}(q).\end{equation}

Let $u\in H^d(G)$ and $M(r)=\mathrm{osc}_{\partial
B_r(A)}u=\max_{\partial B_r(A)}u-\min_{\partial B_r(A)}u.$ By the
maximum principle which is a direct consequence of the definition of
the harmonic function, we have $\max_{\partial
B_r(A)}u=\max_{B_r(A)}u$ and $\min_{\partial
B_r(A)}u=\min_{B_r(A)}u$, so that $M(r)$ is nondecreasing in $r$. To
prove the theorem, it suffices to show that $M(r)=0$ for any large
$r.$ Let $y_r,x_r\in\partial B_r(A)$ satisfy $u(y_r)=\max_{\partial
B_r(A)}u$ and $u(x_r)=\min_{\partial B_r(A)}u.$ Then there exists a
path in $B_r(A)$ such that
$$y_r=z_0\sim z_1\sim \cdots\sim z_l=x_r,$$ where $z_i\in B_r(A)$ for $0\leq i\leq l$ and $l\leq 3D.$ Hence
\begin{eqnarray}\label{FTP2}
M(r)&=&u(y_r)-u(x_r)\leq \sum_{i=0}^{l-1}|\nabla u|(z_i)\nonumber\\
&\leq& C\sum_{i=0}^{l-1}\frac{C(D)}{\sqrt r}\mathrm{osc}_{B_{6r}(z_i)}u\nonumber\\
&\leq& C(D)\frac{\mathrm{osc}_{B_{7r+D}(q)}u}{\sqrt r}\cdot (3D)\nonumber\\
&\leq& C(D)\frac{\mathrm{osc}_{B_{7r+2D}(A)}u}{\sqrt r}\nonumber\\
&\leq& C(D)\frac{M(9r)}{\sqrt r}\ \ \ for\ \ r \geq D,
\end{eqnarray}
where we use (\ref{WGE}) in Corollary \ref{WGEC},
$B_{6r}(z_i)\subset B_{7r+D}(q)\subset B_{7r+2D}(A).$

Let $r\geq R_0(D,\delta):=\left(\frac{C(D)}{\delta}\right)^2,$ for
$\delta<1$ which will be chosen later. Then we have
$\frac{C(D)}{\sqrt r}\leq\delta<1.$ By (\ref{FTP2}), for any $r\geq
R_0(D,\delta),$ we obtain that for $k\geq1,$
$$M(r)\leq \delta M(9r)\leq \delta^kM(9^kr).$$ Since $u\in H^d(G),$ by \eqref{FTP1}, $$M(r)\leq2\max_{B_{r+D}(q)}|u|\leq 2C\big(r+D+1\big)^d.$$
Hence for large $k$ with $9^kr\geq D+1,$
$$M(r)\leq \delta^kM(9^kr)\leq 2C\delta^k(9^kr+D+1)^d\leq C2^{d+1}\delta^k(9^kr)^d=C(d)\left(\frac{1}{2}\right)^kr^d,$$ if we choose $\delta=\frac{1}{2\cdot9^d}.$
Then for any $r\geq R_0(D,\delta)=(C(D)2\cdot9^d)^2,$ we have
$$M(r)\leq C(d)\left(\frac{1}{2}\right)^kr^d.$$ By $k\rightarrow \infty,$ we obtain $M(r)=0,$ which proves the theorem.
\end{proof}

Combining Theorem \ref{PNG1T} and Theorem \ref{PGTF} with Lemma
\ref{GFL1}, we obtain a dimension estimate that does not depend on
the maximal facial degree $D_G$.
\begin{theorem} Let G be a semiplanar graph with $\Sec(G)\geq 0$. Then for any $d\geq 1$,
$$\dim H^d(G)\leq Cd^2,$$ where C is an absolute constant.
\end{theorem}

{\bf Acknowledgements.} The authors of the paper thank the
anonymous referee for his/her excellent work with numerous careful and useful suggestions and comments.

\end{document}